\documentclass{tac} 

\usepackage[colorlinks,linkcolor=darkblue,citecolor=darkblue,urlcolor=darkblue,breaklinks=true,pagebackref]{hyperref}
\usepackage[centertags,sumlimits,intlimits,namelimits,reqno]{amsmath}
\usepackage{latexsym,amsfonts,amssymb,exscale,enumerate,breakcites}
\usepackage{tikz}
\usepackage{float,multirow,array,multicol}
\usepackage[all,cmtip,2cell]{xy}

\usepackage{stmaryrd,mathrsfs}
\usepackage{anyfontsize}

\usepackage{tikz-cd}

\definecolor{darkblue}{rgb}{0,0,0.7}

\newcommand\maps{{\colon}}
\newcommand{\define}[1]{{\bf \boldmath #1}}

  \newcommand{\R}{{\mathbb{R}}}
  \newcommand{\hooklongrightarrow}{\lhook\joinrel\longrightarrow}
  \newcommand{\FinSet}{\mathrm{FinSet}}
  \newcommand{\Set}{\mathrm{Set}}
  \newcommand{\opp}{\mathrm{op}}

  \newcommand{\FinVect}{\mathrm{FinVect}}
  \newcommand{\Vect}{\mathrm{Vect}}
  \newcommand{\LinRel}{\mathrm{LinRel}}
  \DeclareMathOperator\corel{{Corel}}
  \DeclareMathOperator\cospan{{Cospan}}

\newcommand{\mc}{\mathcal}
\newcommand{\ot}{\otimes}
\newcommand{\im}{\mathrm{Im}\,}

\usetikzlibrary{backgrounds,circuits,circuits.ee.IEC,shapes,fit,matrix,automata,decorations.markings}
\UseTwocells

  \pgfdeclarelayer{edgelayer}
  \pgfdeclarelayer{nodelayer}
  \pgfdeclarelayer{background}
  \pgfsetlayers{background,edgelayer,nodelayer,main}

  \tikzset{font=\footnotesize}
  \tikzstyle{none}=[inner sep=0pt]

\tikzstyle{circ}=[circle,fill=black,draw,inner sep=3pt]
\tikzstyle{circ2}=[circle,fill=black,draw,inner sep=1pt]
\tikzstyle{dot}=[circle,fill=black,draw,inner sep=1.5pt]
\tikzstyle{dotbig}=[circle,fill=black,draw,inner sep=2.2pt]
\tikzstyle{sdot}=[circle,fill=white,draw,inner sep=1pt]
\tikzstyle{amp}=[rounded rectangle, rounded rectangle west arc=0pt, draw, inner
sep = 2pt,minimum width=4ex,fill=white]

\tikzstyle{sq}=[rectangle,draw,inner sep=1pt,minimum width=15pt,minimum height=15pt]
  
\tikzstyle{mapsto}=[color=gray,very thick,shorten>=10pt,shorten<=5pt,->,>=stealth]

\tikzstyle{connection}=[circuit symbol open,
    circuit symbol size=width .5 height .5,
    shape=circle ee,
  inner sep=0.25\pgflinewidth]
  \tikzset{->-/.style={decoration={
    markings,
    mark=at position .54 with {\arrow{latex}}},postaction={decorate}}}

\newcommand{\mult}[1]
{
\begin{aligned}
    \resizebox{#1}{!}{
\begin{tikzpicture}
	\begin{pgfonlayer}{nodelayer}
		\node [style=none] (0) at (1, -0) {};
		\node [style=circ] (1) at (0.125, -0) {};
		\node [style=none] (2) at (-1, 0.5) {};
		\node [style=none] (3) at (-1, -0.5) {};
	\end{pgfonlayer}
	\begin{pgfonlayer}{edgelayer}
		\draw[line width=2pt] (0.center) to (1.center);
		\draw[line width=2pt] [in=0, out=120, looseness=1.20] (1.center) to (2.center);
		\draw[line width=2pt] [in=0, out=-120, looseness=1.20] (1.center) to (3.center);
	\end{pgfonlayer}
      \end{tikzpicture}}
\end{aligned}
}

\newcommand{\unit}[1]
{
  \begin{aligned}
    \resizebox{#1}{!}{
\begin{tikzpicture}
	\begin{pgfonlayer}{nodelayer}
		\node [style=none] (0) at (1, -0) {};
		\node [style=none] (1) at (-1, -0) {};
		\node [style=circ] (2) at (0, -0) {};
	\end{pgfonlayer}
	\begin{pgfonlayer}{edgelayer}
		\draw[line width=2pt] (0.center) to (2);
	\end{pgfonlayer}
      \end{tikzpicture}}
  \end{aligned}
}

\newcommand{\comult}[1]
{
\begin{aligned}
    \resizebox{#1}{!}{
\begin{tikzpicture}
	\begin{pgfonlayer}{nodelayer}
		\node [style=none] (0) at (-1, -0) {};
		\node [style=circ] (1) at (-0.125, -0) {};
		\node [style=none] (2) at (1, 0.5) {};
		\node [style=none] (3) at (1, -0.5) {};
	\end{pgfonlayer}
	\begin{pgfonlayer}{edgelayer}
		\draw[line width=2pt] (0.center) to (1.center);
		\draw[line width=2pt] [in=180, out=60, looseness=1.20] (1.center) to (2.center);
		\draw[line width=2pt] [in=180, out=-60, looseness=1.20] (1.center) to (3.center);
	\end{pgfonlayer}
      \end{tikzpicture}}
\end{aligned}
}

\newcommand{\counit}[1]
{
  \begin{aligned}
    \resizebox{#1}{!}{
\begin{tikzpicture}
	\begin{pgfonlayer}{nodelayer}
		\node [style=none] (0) at (-1, -0) {};
		\node [style=none] (1) at (1, -0) {};
		\node [style=circ] (2) at (0, -0) {};
	\end{pgfonlayer}
	\begin{pgfonlayer}{edgelayer}
		\draw[line width=2pt] (0.center) to (2);
	\end{pgfonlayer}
      \end{tikzpicture}}
  \end{aligned}
}

\newcommand{\idone}[1]
{
  \begin{aligned}
    \resizebox{#1}{!}{
     \begin{tikzpicture}
	\begin{pgfonlayer}{nodelayer}
		\node [style=none] (0) at (-1, -0) {};
		\node [style=none] (1) at (1, -0) {};
		\node [style=none] (2) at (0, 0.5) {};
		\node [style=none] (3) at (0, -0.5) {};
	\end{pgfonlayer}
	\begin{pgfonlayer}{edgelayer}
		\draw[line width=2pt] (1.center) to (0.center);
	\end{pgfonlayer}
\end{tikzpicture} 
    }
  \end{aligned}
}

\newcommand{\swap}[1]
{
  \begin{aligned}
    \resizebox{#1}{!}{
\begin{tikzpicture}
	\begin{pgfonlayer}{nodelayer}
		\node [style=none] (2) at (-0.5, -0.5) {};
		\node [style=none] (3) at (-2, 0.5) {};
		\node [style=none] (4) at (-0.5, 0.5) {};
		\node [style=none] (5) at (-2, -0.5) {};
	\end{pgfonlayer}
	\begin{pgfonlayer}{edgelayer}
		\draw[line width=2pt] [in=180, out=0, looseness=1.00] (3.center) to (2.center);
		\draw[line width=2pt] [in=0, out=180, looseness=1.00] (4.center) to (5.center);
	\end{pgfonlayer}
\end{tikzpicture}
    }
  \end{aligned}
}
\newcommand{\assocl}[1]
{
  \begin{aligned}
    \resizebox{#1}{!}{
\begin{tikzpicture}
	\begin{pgfonlayer}{nodelayer}
		\node [style=circ] (0) at (0.125, -0) {};
		\node [style=none] (1) at (-1, 0.5) {};
		\node [style=none] (2) at (-1, -0.5) {};
		\node [style=none] (3) at (0, -1) {};
		\node [style=none] (4) at (2.25, -0.5) {};
		\node [style=none] (5) at (0.25, -0) {};
		\node [style=circ] (6) at (1.25, -0.5) {};
		\node [style=none] (7) at (-1, -1) {};
	\end{pgfonlayer}
	\begin{pgfonlayer}{edgelayer}
		\draw[line width=2pt] [in=0, out=120, looseness=1.20] (0.center) to (1.center);
		\draw[line width=2pt] [in=0, out=-120, looseness=1.20] (0.center) to (2.center);
		\draw[line width=2pt] (4.center) to (6);
		\draw[line width=2pt] [in=0, out=120, looseness=1.20] (6) to (5.center);
		\draw[line width=2pt] [in=0, out=-120, looseness=1.20] (6) to (3.center);
		\draw[line width=2pt] (3.center) to (7.center);
	\end{pgfonlayer}
      \end{tikzpicture}}
  \end{aligned}
}

\newcommand{\assocr}[1]
{
  \begin{aligned}
    \resizebox{#1}{!}{
\begin{tikzpicture}
	\begin{pgfonlayer}{nodelayer}
		\node [style=circ] (0) at (0.125, -0.5) {};
		\node [style=none] (1) at (-1, -1) {};
		\node [style=none] (2) at (-1, 0) {};
		\node [style=none] (3) at (0, 0.5) {};
		\node [style=none] (4) at (2.25, 0) {};
		\node [style=none] (5) at (0.25, -0.5) {};
		\node [style=circ] (6) at (1.25, 0) {};
		\node [style=none] (7) at (-1, 0.5) {};
	\end{pgfonlayer}
	\begin{pgfonlayer}{edgelayer}
		\draw[line width=2pt] [in=0, out=-120, looseness=1.20] (0.center) to (1.center);
		\draw[line width=2pt] [in=0, out=120, looseness=1.20] (0.center) to (2.center);
		\draw[line width=2pt] (4.center) to (6);
		\draw[line width=2pt] [in=0, out=-120, looseness=1.20] (6) to (5.center);
		\draw[line width=2pt] [in=0, out=120, looseness=1.20] (6) to (3.center);
		\draw[line width=2pt] (3.center) to (7.center);
	\end{pgfonlayer}
      \end{tikzpicture}}
  \end{aligned}
}

\newcommand{\coassocl}[1]
{
  \begin{aligned}
    \resizebox{#1}{!}{
\begin{tikzpicture}
	\begin{pgfonlayer}{nodelayer}
		\node [style=circ] (0) at (1.125, -0.5) {};
		\node [style=none] (1) at (2.25, -1) {};
		\node [style=none] (2) at (2.25, 0) {};
		\node [style=none] (3) at (1.25, 0.5) {};
		\node [style=none] (4) at (-1, 0) {};
		\node [style=none] (5) at (1, -0.5) {};
		\node [style=circ] (6) at (0, 0) {};
		\node [style=none] (7) at (2.25, 0.5) {};
	\end{pgfonlayer}
	\begin{pgfonlayer}{edgelayer}
		\draw[line width=2pt] [in=180, out=-60, looseness=1.20] (0.center) to (1.center);
		\draw[line width=2pt] [in=180, out=60, looseness=1.20] (0.center) to (2.center);
		\draw[line width=2pt] (4.center) to (6);
		\draw[line width=2pt] [in=180, out=-60, looseness=1.20] (6) to (5.center);
		\draw[line width=2pt] [in=180, out=60, looseness=1.20] (6) to (3.center);
		\draw[line width=2pt] (3.center) to (7.center);
	\end{pgfonlayer}
      \end{tikzpicture}}
  \end{aligned}
}

\newcommand{\coassocr}[1]
{
  \begin{aligned}
    \resizebox{#1}{!}{
\begin{tikzpicture}
	\begin{pgfonlayer}{nodelayer}
		\node [style=circ] (0) at (1.125, 0) {};
		\node [style=none] (1) at (2.25, 0.5) {};
		\node [style=none] (2) at (2.25, -0.5) {};
		\node [style=none] (3) at (1.25, -1) {};
		\node [style=none] (4) at (-1, -0.5) {};
		\node [style=none] (5) at (1, 0) {};
		\node [style=circ] (6) at (0, -0.5) {};
		\node [style=none] (7) at (2.25, -1) {};
	\end{pgfonlayer}
	\begin{pgfonlayer}{edgelayer}
		\draw[line width=2pt] [in=180, out=60, looseness=1.20] (0.center) to (1.center);
		\draw[line width=2pt] [in=180, out=-60, looseness=1.20] (0.center) to (2.center);
		\draw[line width=2pt] (4.center) to (6);
		\draw[line width=2pt] [in=180, out=60, looseness=1.20] (6) to (5.center);
		\draw[line width=2pt] [in=180, out=-60, looseness=1.20] (6) to (3.center);
		\draw[line width=2pt] (3.center) to (7.center);
	\end{pgfonlayer}
      \end{tikzpicture}}
  \end{aligned}
}

\newcommand{\unitl}[1]
{
  \begin{aligned}
    \resizebox{#1}{!}{
\begin{tikzpicture}
	\begin{pgfonlayer}{nodelayer}
		\node [style=none] (0) at (1, -0) {};
		\node [style=circ] (1) at (0.125, -0) {};
		\node [style=circ] (2) at (-1, 0.5) {};
		\node [style=none] (3) at (-1, -0.5) {};
		\node [style=none] (4) at (-2, -0.5) {};
	\end{pgfonlayer}
	\begin{pgfonlayer}{edgelayer}
		\draw[line width=2pt] (0.center) to (1.center);
		\draw[line width=2pt] [in=0, out=120, looseness=1.20] (1.center) to (2.center);
		\draw[line width=2pt] [in=0, out=-120, looseness=1.20] (1.center) to (3.center);
		\draw[line width=2pt] (4.center) to (3.center);
	\end{pgfonlayer}
\end{tikzpicture}
    }
  \end{aligned}
}

\newcommand{\counitl}[1]
{
  \begin{aligned}
    \resizebox{#1}{!}{
\begin{tikzpicture}
	\begin{pgfonlayer}{nodelayer}
		\node [style=none] (0) at (-2, -0) {};
		\node [style=circ] (1) at (-1.125, -0) {};
		\node [style=circ] (2) at (0, 0.5) {};
		\node [style=none] (3) at (0, -0.5) {};
		\node [style=none] (4) at (1, -0.5) {};
	\end{pgfonlayer}
	\begin{pgfonlayer}{edgelayer}
		\draw[line width=2pt] (0.center) to (1.center);
		\draw[line width=2pt] [in=180, out=60, looseness=1.20] (1.center) to (2.center);
		\draw[line width=2pt] [in=180, out=-60, looseness=1.20] (1.center) to (3.center);
		\draw[line width=2pt] (4.center) to (3.center);
	\end{pgfonlayer}
\end{tikzpicture}
    }
  \end{aligned}
}

\newcommand{\commute}[1]
{
  \begin{aligned}
    \resizebox{#1}{!}{
\begin{tikzpicture}
	\begin{pgfonlayer}{nodelayer}
		\node [style=none] (0) at (1.25, -0) {};
		\node [style=circ] (1) at (0.375, -0) {};
		\node [style=none] (2) at (-0.5, -0.5) {};
		\node [style=none] (3) at (-2, 0.5) {};
		\node [style=none] (4) at (-0.5, 0.5) {};
		\node [style=none] (5) at (-2, -0.5) {};
	\end{pgfonlayer}
	\begin{pgfonlayer}{edgelayer}
		\draw[line width=2pt] (0.center) to (1.center);
		\draw[line width=2pt] [in=0, out=-120, looseness=1.20] (1.center) to (2.center);
		\draw[line width=2pt] [in=180, out=0, looseness=1.00] (3.center) to (2.center);
		\draw[line width=2pt] [in=0, out=120, looseness=1.20] (1.center) to (4.center);
		\draw[line width=2pt] [in=0, out=180, looseness=1.00] (4.center) to (5.center);
	\end{pgfonlayer}
\end{tikzpicture}
    }
  \end{aligned}
}

\newcommand{\cocommute}[1]
{
  \begin{aligned}
    \resizebox{#1}{!}{
\begin{tikzpicture}
	\begin{pgfonlayer}{nodelayer}
		\node [style=none] (0) at (-2, -0) {};
		\node [style=circ] (1) at (-1.125, -0) {};
		\node [style=none] (2) at (-0.25, -0.5) {};
		\node [style=none] (3) at (1.25, 0.5) {};
		\node [style=none] (4) at (-0.25, 0.5) {};
		\node [style=none] (5) at (1.25, -0.5) {};
	\end{pgfonlayer}
	\begin{pgfonlayer}{edgelayer}
		\draw[line width=2pt] (0.center) to (1.center);
		\draw[line width=2pt] [in=180, out=-60, looseness=1.20] (1.center) to (2.center);
		\draw[line width=2pt] [in=0, out=180, looseness=1.00] (3.center) to (2.center);
		\draw[line width=2pt] [in=180, out=60, looseness=1.20] (1.center) to (4.center);
		\draw[line width=2pt] [in=180, out=0, looseness=1.00] (4.center) to (5.center);
	\end{pgfonlayer}
\end{tikzpicture}
    }
  \end{aligned}
}

\newcommand{\frobs}[1]
{
  \begin{aligned}
    \resizebox{#1}{!}{
\begin{tikzpicture}
	\begin{pgfonlayer}{nodelayer}
		\node [style=none] (0) at (-1.5, 0.5) {};
		\node [style=circ] (1) at (-0.75, 0.5) {};
		\node [style=none] (2) at (0.25, -0) {};
		\node [style=none] (3) at (0.25, 1) {};
		\node [style=circ] (4) at (1, -0.5) {};
		\node [style=none] (5) at (0, -0) {};
		\node [style=none] (6) at (1.75, -0.5) {};
		\node [style=none] (7) at (0, -1) {};
		\node [style=none] (8) at (1.75, 1) {};
		\node [style=none] (9) at (-1.5, -1) {};
	\end{pgfonlayer}
	\begin{pgfonlayer}{edgelayer}
		\draw[line width=2pt] [in=180, out=-60, looseness=1.20] (1) to (2.center);
		\draw[line width=2pt] [in=180, out=60, looseness=1.20] (1) to (3.center);
		\draw[line width=2pt] (0.center) to (1);
		\draw[line width=2pt] (6.center) to (4);
		\draw[line width=2pt] [in=0, out=120, looseness=1.20] (4) to (5.center);
		\draw[line width=2pt] [in=0, out=-120, looseness=1.20] (4) to (7.center);
		\draw[line width=2pt] (3.center) to (8.center);
		\draw[line width=2pt] (7.center) to (9.center);
	\end{pgfonlayer}
\end{tikzpicture}
    }
  \end{aligned}
}

\newcommand{\frobx}[1]
{
  \begin{aligned}
    \resizebox{#1}{!}{
\begin{tikzpicture}
	\begin{pgfonlayer}{nodelayer}
		\node [style=circ] (0) at (-0.5, -0) {};
		\node [style=none] (1) at (-1.5, -0.5) {};
		\node [style=none] (2) at (-1.5, 0.5) {};
		\node [style=circ] (3) at (0.5, -0) {};
		\node [style=none] (4) at (1.5, -0.5) {};
		\node [style=none] (5) at (1.5, 0.5) {};
	\end{pgfonlayer}
	\begin{pgfonlayer}{edgelayer}
		\draw[line width=2pt, in=0, out=-120, looseness=1.20] (0.center) to (1.center);
		\draw[line width=2pt, in=0, out=120, looseness=1.20] (0.center) to (2.center);
		\draw[line width=2pt, in=180, out=-60, looseness=1.20] (3) to (4.center);
		\draw[line width=2pt, in=180, out=60, looseness=1.20] (3) to (5.center);
		\draw[line width=2pt] (0) to (3);
	\end{pgfonlayer}
\end{tikzpicture}
    }
  \end{aligned}
}

\newcommand{\frobz}[1]
{
  \begin{aligned}
    \resizebox{#1}{!}{
\begin{tikzpicture}
	\begin{pgfonlayer}{nodelayer}
		\node [style=none] (0) at (1.75, 0.5) {};
		\node [style=circ] (1) at (1, 0.5) {};
		\node [style=none] (2) at (0, -0) {};
		\node [style=none] (3) at (0, 1) {};
		\node [style=circ] (4) at (-0.75, -0.5) {};
		\node [style=none] (5) at (0.25, -0) {};
		\node [style=none] (6) at (-1.5, -0.5) {};
		\node [style=none] (7) at (0.25, -1) {};
		\node [style=none] (8) at (-1.5, 1) {};
		\node [style=none] (9) at (1.75, -1) {};
	\end{pgfonlayer}
	\begin{pgfonlayer}{edgelayer}
		\draw[line width=2pt] [in=0, out=-120, looseness=1.20] (1) to (2.center);
		\draw[line width=2pt] [in=0, out=120, looseness=1.20] (1) to (3.center);
		\draw[line width=2pt] (0.center) to (1);
		\draw[line width=2pt] (6.center) to (4);
		\draw[line width=2pt] [in=180, out=60, looseness=1.20] (4) to (5.center);
		\draw[line width=2pt] [in=180, out=-60, looseness=1.20] (4) to (7.center);
		\draw[line width=2pt] (3.center) to (8.center);
		\draw[line width=2pt] (7.center) to (9.center);
	\end{pgfonlayer}
\end{tikzpicture}
    }
  \end{aligned}
}

\newcommand{\spec}[1]
{
  \begin{aligned}
    \resizebox{#1}{!}{
\begin{tikzpicture}
	\begin{pgfonlayer}{nodelayer}
		\node [style=none] (0) at (1.75, -0) {};
		\node [style=circ] (1) at (0.75, -0) {};
		\node [style=none] (2) at (0, -0.5) {};
		\node [style=none] (3) at (0, 0.5) {};
		\node [style=circ] (4) at (-0.75, -0) {};
		\node [style=none] (5) at (0, -0.5) {};
		\node [style=none] (6) at (-1.75, -0) {};
		\node [style=none] (7) at (0, 0.5) {};
	\end{pgfonlayer}
	\begin{pgfonlayer}{edgelayer}
		\draw[line width=2pt] (0.center) to (1.center);
		\draw[line width=2pt] [in=0, out=-120, looseness=1.20] (1.center) to (2.center);
		\draw[line width=2pt] [in=0, out=120, looseness=1.20] (1.center) to (3.center);
		\draw[line width=2pt] (6.center) to (4);
		\draw[line width=2pt] [in=180, out=-60, looseness=1.20] (4) to (5.center);
		\draw[line width=2pt] [in=180, out=60, looseness=1.20] (4) to (7.center);
	\end{pgfonlayer}
\end{tikzpicture}
    }
  \end{aligned}
}

  \author{Brendan Fong}

  \thanks{I thank John Baez, Sam Staton, and Aleks Kissinger for useful
    conversations. I also gratefully acknowledge the support of the Queen
    Elizabeth Scholarships, Oxford, and the Basic Research Office of the ASDR\&E
    through ONR N00014-16-1-2010.
  }

  \address{Department of Electrical and Systems Engineering\\
    University of Pennsylvania \\
    United States of America
  }

  \title{Decorated corelations}

  \copyrightyear{2017}

  \eaddress{fo@seas.upenn.edu}

\begin{document}   

\maketitle

\begin{abstract}
  Let $\mathcal C$ be a category with finite colimits, and let $(\mc E,\mc M)$
  be a factorisation system on $\mc C$ with $\mc M$ stable under pushouts.
  Writing $\mc C;\mc M^\opp$ for the symmetric monoidal category with morphisms
  cospans of the form $\stackrel{c}\to \stackrel{m}\leftarrow$, where $c \in \mc
  C$ and $m \in \mc M$, we give method for constructing a category from a
  symmetric lax monoidal functor $F\maps (\mc C; \mc M^\opp,+) \to
  (\Set,\times)$. A morphism in this category, termed a \emph{decorated
  corelation}, comprises (i) a cospan $X \to N \leftarrow Y$ in $\mc C$ such
  that the canonical copairing $X+Y \to N$ lies in $\mc E$, together with (ii)
  an element of $FN$.  Functors between decorated corelation categories can be
  constructed from natural transformations between the decorating functors $F$.
  This provides a general method for constructing hypergraph
  categories---symmetric monoidal categories in which each object is a special
  commutative Frobenius monoid in a coherent way---and their functors. Such
  categories are useful for modelling network languages, for example circuit
  diagrams, and such functors their semantics.
\end{abstract}

\section{Introduction}

Consider a circuit diagram.
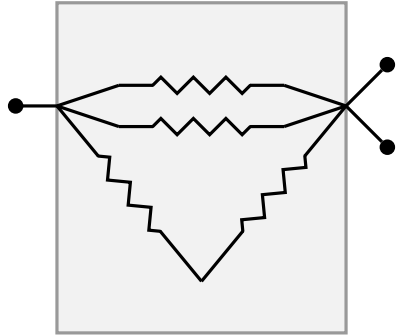
\[
\resizebox{.35\textwidth}{!}{
    \tikzset{every path/.style={line width=1.1pt}}
  \begin{tikzpicture}[circuit ee IEC, set resistor graphic=var resistor IEC graphic]
	\begin{pgfonlayer}{nodelayer}
		\node [style=dot] (0) at (-2.5, 0.75) {};
		\coordinate (1) at (-2, -2) {};
		\coordinate (2) at (1.5, 2) {};
		\coordinate (A) at (-2, 0.75) {};
		\node [style=dot] (4) at (2, 1.25) {};
		\coordinate (B) at (1.5, 0.75) {};
		\node [style=dot] (6) at (2, 0.25) {};
		\coordinate (ub) at (0.75, 1) {};
		\coordinate (la) at (-1.25, 0.5) {};
		\coordinate (C) at (-0.25, -1.375) {};
		\coordinate (lb) at (0.75, 0.5) {};
		\coordinate (ua) at (-1.25, 1) {};
	\end{pgfonlayer}
	\begin{pgfonlayer}{edgelayer}
		\draw (0.center) to (A);
		\draw (6) to (B);
		\draw (4) to (B);
    \path (A) edge (ua);
    \path (A) edge (la);
    \path (B) edge (ub);
    \path (B) edge (lb);
    \path (ua) edge  [resistor] (ub);
    \path (la) edge  [resistor] (lb);
    \path (A) edge  [resistor]  (C);
    \path (C) edge  [resistor]  (B);
	\end{pgfonlayer}
	\begin{pgfonlayer}{background}
	  \filldraw [fill=black!5!white, draw=black!40!white] (1) rectangle (2);
	\end{pgfonlayer}
\end{tikzpicture}
}
\]
We often view such diagrams atomically, representing the complete physical
system built as specified. Yet the very process of building such a system
involves assembling it from its parts, each of which we might diagram in the
same way. The goal of this paper is to develop formal category-theoretic tools
for describing and interpreting this process of assembly.

As we wish to compose circuits, we model them as morphisms in a category. One
method for realising the above circuit as a morphism is to use decorated cospans
\cite{Fon15}. To do so, consider the part inside the shaded area as a graph with
three vertices and the four resistors as edges. Writing $n$ for a set of $n$
elements, we have functions $1 \to 3$ and $2 \to 3$ describing how the terminals
$\bullet$ on the left and the right respectively are attached to the vertex set
$3$ of this graph. Thus the above circuit can be modelled as a \emph{cospan} of
functions $1 \to 3 \leftarrow 2$, \emph{decorated} by the aforementioned graph
on the apex $3$ of this cospan.

While often useful for syntactic purposes, a significant limitation of using
cospans alone is that composition of cospans indiscrimately accumulates
information. For example, here is a depiction of the composite of five circuits
using decorated cospans:
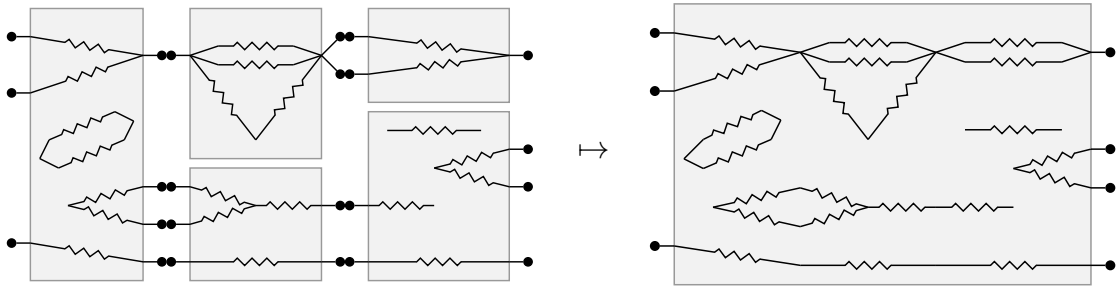
\[
\begin{aligned}
\resizebox{.45\textwidth}{!}{
    \tikzset{every path/.style={line width=1.1pt}}
  \begin{tikzpicture}[circuit ee IEC, set resistor graphic=var resistor IEC graphic]
	\begin{pgfonlayer}{nodelayer}
		\node [style=dotbig] (0) at (-2.25, 4) {};
		\coordinate (1) at (-1.75, 1.25) {};
		\coordinate (2) at (1.75, 5.25) {};
		\coordinate (3) at (-1.75, 4) {};
		\node [style=dotbig] (4) at (2.25, 4.5) {};
		\coordinate (5) at (1.75, 4) {};
		\node [style=dotbig] (6) at (2.25, 3.5) {};
		\coordinate (7) at (1, 4.25) {};
		\coordinate (8) at (-1, 3.75) {};
		\coordinate (9) at (0, 1.75) {};
		\coordinate (10) at (1, 3.75) {};
		\coordinate (11) at (-1, 4.25) {};
		\node [style=dotbig] (12) at (7.25, 4) {};
		\coordinate (13) at (-3, -1.5) {};
		\coordinate (14) at (6, 2) {};
		\coordinate (15) at (3.5, 2) {};
		\coordinate (16) at (4.75, -0) {};
		\coordinate (17) at (3, 3.5) {};
		\coordinate (18) at (3, 4.5) {};
		\node [style=dotbig] (19) at (2.5, -1.5) {};
		\coordinate (20) at (6.75, 4) {};
		\node [style=dotbig] (21) at (7.25, -1.5) {};
		\node [style=dotbig] (22) at (-6.5, -1) {};
		\node [style=dotbig] (23) at (7.25, 1.5) {};
		\node [style=dotbig] (24) at (-6.5, 4.5) {};
		\node [style=dotbig] (25) at (-6.5, 3) {};
		\coordinate (26) at (-6, -2) {};
		\coordinate (27) at (-3, 5.25) {};
		\coordinate (28) at (1.75, 1) {};
		\coordinate (29) at (-1.75, -2) {};
		\coordinate (30) at (3, -2) {};
		\coordinate (31) at (-6, 4.5) {};
		\coordinate (32) at (-6, 3) {};
		\coordinate (33) at (-3, 4) {};
		\coordinate (34) at (6.75, 2.5) {};
		\coordinate (35) at (3, 2.75) {};
		\coordinate (36) at (6.75, 5.25) {};
		\node [style=dotbig] (37) at (7.25, 0.5) {};
		\node [style=dotbig] (38) at (-2.25, -1.5) {};
		\node [style=dotbig] (39) at (-2.5, 4) {};
		\node [style=dotbig] (40) at (2.5, 4.5) {};
		\node [style=dotbig] (41) at (2.5, 3.5) {};
		\node [style=dotbig] (42) at (-2.25, 0.5) {};
		\node [style=dotbig] (43) at (-2.5, -1.5) {};
		\node [style=dotbig] (44) at (-2.5, 0.5) {};
		\node [style=dotbig] (45) at (2.25, -1.5) {};
		\node [style=dotbig] (46) at (2.25, -0) {};
		\node [style=dotbig] (47) at (2.5, -0) {};
		\node [style=dotbig] (48) at (-2.5, -0.5) {};
		\node [style=dotbig] (49) at (-2.25, -0.5) {};
		\coordinate (50) at (-6, -1) {};
		\coordinate (51) at (-5, -0) {};
		\coordinate (52) at (-3, 0.5) {};
		\coordinate (53) at (-3, -0.5) {};
		\coordinate (54) at (-1.75, 0.5) {};
		\coordinate (55) at (-1.75, -0.5) {};
		\coordinate (56) at (0, -0) {};
		\coordinate (57) at (1.75, -0) {};
		\coordinate (58) at (3, -0) {};
		\coordinate (59) at (-5.75, 1.25) {};
		\coordinate (60) at (4.75, 1) {};
		\coordinate (61) at (6.75, 1.5) {};
		\coordinate (62) at (6.75, 0.5) {};
		\coordinate (63) at (-1.75, -1.5) {};
		\coordinate (64) at (1.75, -1.5) {};
		\coordinate (65) at (3, -1.5) {};
		\coordinate (66) at (6.75, -1.5) {};
		\coordinate (67) at (-5.5, 1.75) {};
		\coordinate (68) at (-5.25, 1) {};
		\coordinate (69) at (-3.75, 2.5) {};
		\coordinate (70) at (-3.5, 1.75) {};
		\coordinate (71) at (-3.25, 2.25) {};
	\end{pgfonlayer}
	\begin{pgfonlayer}{edgelayer}
		\path (50) edge [resistor] (13);
		\draw (24) to (31);
		\draw (61) to (23);
		\draw (62) to (37);
		\draw (25) to (32);
		\draw (22) to (50);
		\draw (33) to (39);
		\draw (0) to (3);
		\draw (52) to (44);
		\draw (42) to (54);
		\draw (53) to (48);
		\draw (49) to (55);
		\path (31) edge [resistor] (33);
		\path (32) edge [resistor] (33);
		\draw (3) to (11);
		\draw (3) to (8);
		\path (11) edge [resistor] (7);
		\path (8) edge [resistor] (10);
		\draw (7) to (5);
		\draw (10) to (5);
		\draw (5) to (4);
		\draw (5) to (6);
		\draw (41) to (17);
		\draw (40) to (18);
		\path (18) edge [resistor] (20);
		\path (17) edge [resistor] (20);
		\draw (20) to (12);
		\path (15) edge [resistor] (14);
		\path (54) edge [resistor] (56);
		\path (55) edge [resistor] (56);
		\path (56) edge [resistor] (57);
		\path (51) edge [resistor] (52);
		\path (51) edge [resistor] (53);
		\draw (57) to (46);
		\draw (47) to (58);
		\path (58) edge [resistor] (16);
		\path (3) edge [resistor] (9);
		\path (9) edge [resistor] (5);
		\draw (13) to (43);
		\path (60) edge [resistor] (61);
		\path (60) edge [resistor] (62);
		\draw (38) to (63);
		\path (63) edge [resistor] (64);
		\draw (64) to (45);
		\draw (19) to (65);
		\path (65) edge [resistor] (66);
		\draw (66) to (21);
		\path (67) edge [resistor] (69);
		\path (68) edge [resistor] (70);
		\draw (59) to (67);
		\draw (59) to (68);
		\draw (69) to (71);
		\draw (71) to (70);
	\end{pgfonlayer}
	\begin{pgfonlayer}{background}
	  \filldraw [fill=black!5!white, draw=black!40!white] (1) rectangle (2);
	  \filldraw [fill=black!5!white, draw=black!40!white] (26) rectangle (27);
	  \filldraw [fill=black!5!white, draw=black!40!white] (29) rectangle (28);
	  \filldraw [fill=black!5!white, draw=black!40!white] (35) rectangle (36);
	  \filldraw [fill=black!5!white, draw=black!40!white] (30) rectangle (34);
	\end{pgfonlayer}
\end{tikzpicture}
}
\end{aligned}
\quad
\mapsto
\enspace
\begin{aligned}
\resizebox{.4\textwidth}{!}{
    \tikzset{every path/.style={line width=1.1pt}}
  \begin{tikzpicture}[circuit ee IEC, set resistor graphic=var resistor IEC graphic]
	\begin{pgfonlayer}{nodelayer}
		\coordinate (0) at (-1.75, 4) {};
		\coordinate (1) at (1.75, 4) {};
		\coordinate (2) at (1, 4.25) {};
		\coordinate (3) at (-1, 3.75) {};
		\coordinate (4) at (0, 1.75) {};
		\coordinate (5) at (1, 3.75) {};
		\coordinate (6) at (-1, 4.25) {};
		\node [style=dotbig] (7) at (6.25, 4) {};
		\coordinate (8) at (5, 2) {};
		\coordinate (9) at (2.5, 2) {};
		\coordinate (10) at (3.75, -0) {};
		\coordinate (11) at (2.5, 3.75) {};
		\coordinate (12) at (2.5, 4.25) {};
		\coordinate (13) at (5.75, 4) {};
		\node [style=dotbig] (14) at (6.25, -1.5) {};
		\node [style=dotbig] (15) at (-5.5, -1) {};
		\node [style=dotbig] (16) at (6.25, 1.5) {};
		\node [style=dotbig] (17) at (-5.5, 4.5) {};
		\node [style=dotbig] (18) at (-5.5, 3) {};
		\coordinate (19) at (-5, -2) {};
		\coordinate (20) at (-5, 4.5) {};
		\coordinate (21) at (-5, 3) {};
		\coordinate (22) at (5.75, 5.25) {};
		\node [style=dotbig] (23) at (6.25, 0.5) {};
		\coordinate (24) at (-5, -1) {};
		\coordinate (25) at (-4, -0) {};
		\coordinate (26) at (-1.75, 0.5) {};
		\coordinate (27) at (-1.75, -0.5) {};
		\coordinate (28) at (0, -0) {};
		\coordinate (29) at (1.75, -0) {};
		\coordinate (30) at (-4.75, 1.25) {};
		\coordinate (31) at (3.75, 1) {};
		\coordinate (32) at (5.75, 1.5) {};
		\coordinate (33) at (5.75, 0.5) {};
		\coordinate (34) at (-1.75, -1.5) {};
		\coordinate (35) at (1.75, -1.5) {};
		\coordinate (36) at (5.75, -1.5) {};
		\coordinate (37) at (-4.25, 1.75) {};
		\coordinate (38) at (-4.25, 1) {};
		\coordinate (39) at (-2.75, 2.5) {};
		\coordinate (40) at (-2.5, 1.75) {};
		\coordinate (41) at (-2.25, 2.25) {};
		\coordinate (42) at (5, 4.25) {};
		\coordinate (43) at (5, 3.75) {};
	\end{pgfonlayer}
	\begin{pgfonlayer}{edgelayer}
		\draw (17) to (20);
		\draw (18) to (21);
		\draw (15) to (24);
		\draw (13) to (7);
		\draw (36) to (14);
		\draw (32) to (16);
		\draw (33) to (23);
		\draw (1) to (12);
		\draw (1) to (11);
		\draw (2) to (1);
		\draw (5) to (1);
		\draw (0) to (6);
		\draw (0) to (3);
		\draw (30) to (37);
		\draw (30) to (38);
		\draw (39) to (41);
		\draw (41) to (40);
		\draw (42) to (13);
		\draw (43) to (13);
		\path (6) edge [resistor] (2);
		\path (3) edge [resistor] (5);
		\path (9) edge [resistor] (8);
		\path (26) edge [resistor] (28);
		\path (27) edge [resistor] (28);
		\path (28) edge [resistor] (29);
		\path (0) edge [resistor] (4);
		\path (4) edge [resistor] (1);
		\path (31) edge [resistor] (32);
		\path (31) edge [resistor] (33);
		\path (34) edge [resistor] (35);
		\path (37) edge [resistor] (39);
		\path (38) edge [resistor] (40);
		\path (20) edge [resistor] (0);
		\path (21) edge [resistor] (0);
		\path (25) edge [resistor] (26);
		\path (25) edge [resistor] (27);
		\path (24) edge [resistor] (34);
		\path (29) edge [resistor] (10);
		\path (35) edge [resistor] (36);
		\path (12) edge [resistor] (42);
		\path (11) edge [resistor] (43);
	\end{pgfonlayer}
	\begin{pgfonlayer}{background}
	  \filldraw [fill=black!5!white, draw=black!40!white] (19) rectangle (22);
	\end{pgfonlayer}
\end{tikzpicture}
}
\end{aligned}
\]
Note in particular that the composite of these circuits contains a unique
resistor for every resistor in the factors. If we are interested in describing
the syntax of a diagrammatic language, then this is useful: composition builds
given expressions into a larger one. If we are only interested in the
semantics---given, say by the electrical behaviour at the terminals---this is
often unnecessary and thus often wildly inefficient.

Indeed, suppose our semantics for open circuits is given by the information that
can be gleaned by connecting other open circuits, such as measurement devices,
to the terminals. In these semantics we consider two open circuits equivalent
if, should they be encased, but for their terminals, in a black box
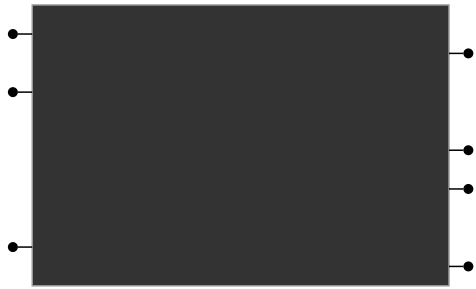
\[
\resizebox{.4\textwidth}{!}{
    \tikzset{every path/.style={line width=1.1pt}}
  \begin{tikzpicture}
    \begin{pgfonlayer}{nodelayer}
		\node [style=dotbig] (14) at (-5.5, 4.5) {};
		\node [style=dotbig] (15) at (-5.5, 3) {};
		\node [style=dotbig] (12) at (-5.5, -1) {};
		\node [style=dotbig] (7) at (6.25, 4) {};
		\node [style=dotbig] (13) at (6.25, 1.5) {};
		\node [style=dotbig] (20) at (6.25, 0.5) {};
		\node [style=dotbig] (11) at (6.25, -1.5) {};
		\coordinate (17) at (-5, 4.5) {};
		\coordinate (18) at (-5, 3) {};
		\coordinate (21) at (-5, -1) {};
		\coordinate (10) at (5.75, 4) {};
		\coordinate (23) at (5.75, 1.5) {};
		\coordinate (24) at (5.75, 0.5) {};
		\coordinate (27) at (5.75, -1.5) {};
		\coordinate (16) at (-5, -2) {};
		\coordinate (19) at (5.75, 5.25) {};
	\end{pgfonlayer}
	\begin{pgfonlayer}{edgelayer}
		\draw (14) to (17);
		\draw (15) to (18);
		\draw (12) to (21);
		\draw (10) to (7);
		\draw (27) to (11);
		\draw (23) to (13);
		\draw (24) to (20);
	\end{pgfonlayer}
	\begin{pgfonlayer}{background}
	  \filldraw [fill=black!80!white, draw=black!40!white] (16) rectangle (19);
	\end{pgfonlayer}
\end{tikzpicture}
}
\]
we would be unable to distinguish them through our electrical investigations. In
this case, at the very least, the previous circuit is equivalent to the circuit
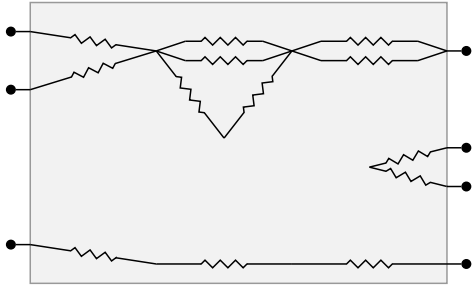
\[
\resizebox{.4\textwidth}{!}{
    \tikzset{every path/.style={line width=1.1pt}}
  \begin{tikzpicture}[circuit ee IEC, set resistor graphic=var resistor IEC graphic]
	\begin{pgfonlayer}{nodelayer}
		\coordinate (0) at (-1.75, 4) {};
		\coordinate (1) at (1.75, 4) {};
		\coordinate (2) at (1, 4.25) {};
		\coordinate (3) at (-1, 3.75) {};
		\coordinate (4) at (0, 1.75) {};
		\coordinate (5) at (1, 3.75) {};
		\coordinate (6) at (-1, 4.25) {};
		\node [style=dotbig] (7) at (6.25, 4) {};
		\coordinate (8) at (2.5, 3.75) {};
		\coordinate (9) at (2.5, 4.25) {};
		\coordinate (10) at (5.75, 4) {};
		\node [style=dotbig] (11) at (6.25, -1.5) {};
		\node [style=dotbig] (12) at (-5.5, -1) {};
		\node [style=dotbig] (13) at (6.25, 1.5) {};
		\node [style=dotbig] (14) at (-5.5, 4.5) {};
		\node [style=dotbig] (15) at (-5.5, 3) {};
		\coordinate (16) at (-5, -2) {};
		\coordinate (17) at (-5, 4.5) {};
		\coordinate (18) at (-5, 3) {};
		\coordinate (19) at (5.75, 5.25) {};
		\node [style=dotbig] (20) at (6.25, 0.5) {};
		\coordinate (21) at (-5, -1) {};
		\coordinate (22) at (3.75, 1) {};
		\coordinate (23) at (5.75, 1.5) {};
		\coordinate (24) at (5.75, 0.5) {};
		\coordinate (25) at (-1.75, -1.5) {};
		\coordinate (26) at (1.75, -1.5) {};
		\coordinate (27) at (5.75, -1.5) {};
		\coordinate (28) at (5, 4.25) {};
		\coordinate (29) at (5, 3.75) {};
	\end{pgfonlayer}
	\begin{pgfonlayer}{edgelayer}
		\draw (14) to (17);
		\draw (15) to (18);
		\draw (12) to (21);
		\draw (10) to (7);
		\draw (27) to (11);
		\draw (23) to (13);
		\draw (24) to (20);
		\draw (1) to (9);
		\draw (1) to (8);
		\draw (2) to (1);
		\draw (5) to (1);
		\draw (0) to (6);
		\draw (0) to (3);
		\draw (28) to (10);
		\draw (29) to (10);
		\path (6) edge [resistor] (2);
		\path (3) edge [resistor] (5);
		\path (0) edge [resistor] (4);
		\path (4) edge [resistor] (1);
		\path (22) edge [resistor] (23);
		\path (22) edge [resistor] (24);
		\path (25) edge [resistor] (26);
		\path (17) edge [resistor] (0);
		\path (18) edge [resistor] (0);
		\path (21) edge [resistor] (25);
		\path (26) edge [resistor] (27);
		\path (9) edge [resistor] (28);
		\path (8) edge [resistor] (29);
	\end{pgfonlayer}
	\begin{pgfonlayer}{background}
	  \filldraw [fill=black!5!white, draw=black!40!white] (16) rectangle (19);
	\end{pgfonlayer}
\end{tikzpicture}
}
\]
where we have removed circuitry not connected to the terminals.  Moreover, this
second circuit is a more efficient representation, as it does not model
inaccessible, internal structure. If we wish to construct a category modelling
the semantics of open circuits then, we require circuit representations and a
composition rule that only retain the information relevant to the black boxed
circuit. In this paper we introduce the notion of corelation to play this role.

Indeed, corelations allow us to pursue a notion of composition that discards
extraneous information as we compose our systems. Consider, for example, the
category $\cospan(\FinSet)$ of cospans in the category of finite sets and
functions. Given a pair of cospans $X \to N \leftarrow Y$, $Y \to M \leftarrow
Z$, their composite has apex the pushout $N+_YM$. This, roughly speaking, is the
union of $N$ and $M$ with two points identified if they are both images of the
same element of $Y$. For example, the following pair of cospans:
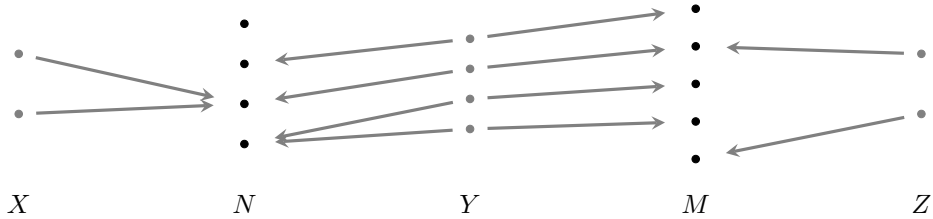
\begin{center}
  \begin{tikzpicture}[auto,scale=2]
    \node[circle,draw,inner sep=1pt,fill=gray,color=gray] (x1) at (-1.5,.2) {};
    \node[circle,draw,inner sep=1pt,fill=gray,color=gray] (x2) at (-1.5,-.2) {};
    \node at (-1.5,-.8) {$X$};
    \node[circle,draw,inner sep=1pt,fill]         (A) at (0,.4) {};
    \node[circle,draw,inner sep=1pt,fill]         (B) at (0,.133) {};
    \node[circle,draw,inner sep=1pt,fill]         (C) at (0,-.133) {};
    \node[circle,draw,inner sep=1pt,fill]         (D) at (0,-.4) {};
    \node at (0,-.8) {$N$};
    \node[circle,draw,inner sep=1pt,fill=gray,color=gray] (y1) at (1.5,.3) {};
    \node[circle,draw,inner sep=1pt,fill=gray,color=gray] (y2) at (1.5,.1) {};
    \node[circle,draw,inner sep=1pt,fill=gray,color=gray] (y3) at (1.5,-.1) {};
    \node[circle,draw,inner sep=1pt,fill=gray,color=gray] (y4) at (1.5,-.3) {};
    \node at (1.5,-.8) {$Y$};
    \node[circle,draw,inner sep=1pt,fill]         (A') at (3,.5) {};
    \node[circle,draw,inner sep=1pt,fill]         (B') at (3,.25) {};
    \node[circle,draw,inner sep=1pt,fill]         (C') at (3,0) {};
    \node[circle,draw,inner sep=1pt,fill]         (D') at (3,-.25) {};
    \node[circle,draw,inner sep=1pt,fill]         (E') at (3,-.5) {};
    \node at (3,-.8) {$M$};
    \node[circle,draw,inner sep=1pt,fill=gray,color=gray] (z1) at (4.5,.2) {};
    \node[circle,draw,inner sep=1pt,fill=gray,color=gray] (z2) at (4.5,-.2) {};
    \node at (4.5,-.8) {$Z$};
    \path[color=gray, very thick, shorten >=10pt, shorten <=5pt, ->, >=stealth]
    (x1) edge (C);
    \path[color=gray, very thick, shorten >=10pt, shorten <=5pt, ->, >=stealth]
    (x2) edge (C);
    \path[color=gray, very thick, shorten >=10pt, shorten <=5pt, ->, >=stealth]
    (y1) edge (B);
    \path[color=gray, very thick, shorten >=10pt, shorten <=5pt, ->, >=stealth]
    (y2) edge (C);
    \path[color=gray, very thick, shorten >=10pt, shorten <=5pt, ->, >=stealth]
    (y3) edge (D);
    \path[color=gray, very thick, shorten >=10pt, shorten <=5pt, ->, >=stealth]
    (y4) edge (D);
    \path[color=gray, very thick, shorten >=10pt, shorten <=5pt, ->, >=stealth]
    (y1) edge (A');
    \path[color=gray, very thick, shorten >=10pt, shorten <=5pt, ->, >=stealth]
    (y2) edge (B');
    \path[color=gray, very thick, shorten >=10pt, shorten <=5pt, ->, >=stealth]
    (y3) edge (C');
    \path[color=gray, very thick, shorten >=10pt, shorten <=5pt, ->, >=stealth]
    (y4) edge (D');
    \path[color=gray, very thick, shorten >=10pt, shorten <=5pt, ->, >=stealth]
    (z1) edge (B');
    \path[color=gray, very thick, shorten >=10pt, shorten <=5pt, ->, >=stealth]
    (z2) edge (E');
  \end{tikzpicture}
\end{center}
becomes
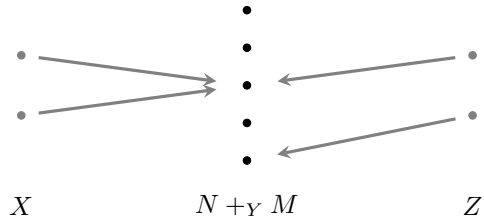
\begin{center}
  \begin{tikzpicture}[auto,scale=2]
    \node[circle,draw,inner sep=1pt,fill=gray,color=gray] (x1) at (1.5,.2) {};
    \node[circle,draw,inner sep=1pt,fill=gray,color=gray] (x2) at (1.5,-.2) {};
    \node at (1.5,-.8) {$X$};
    \node[circle,draw,inner sep=1pt,fill]         (A) at (3,.5) {};
    \node[circle,draw,inner sep=1pt,fill]         (B) at (3,.25) {};
    \node[circle,draw,inner sep=1pt,fill]         (C) at (3,0) {};
    \node[circle,draw,inner sep=1pt,fill]         (D) at (3,-.25) {};
    \node[circle,draw,inner sep=1pt,fill]         (E) at (3,-.5) {};
    \node at (3,-.8) {$N+_YM$};
    \node[circle,draw,inner sep=1pt,fill=gray,color=gray] (z1) at (4.5,.2) {};
    \node[circle,draw,inner sep=1pt,fill=gray,color=gray] (z2) at (4.5,-.2) {};
    \node at (4.5,-.8) {$Z$};
    \path[color=gray, very thick, shorten >=10pt, shorten <=5pt, ->, >=stealth]
    (x1) edge (C);
    \path[color=gray, very thick, shorten >=10pt, shorten <=5pt, ->, >=stealth]
    (x2) edge (C);
    \path[color=gray, very thick, shorten >=10pt, shorten <=5pt, ->, >=stealth]
    (z1) edge (C);
    \path[color=gray, very thick, shorten >=10pt, shorten <=5pt, ->, >=stealth]
    (z2) edge (E);
  \end{tikzpicture}
\end{center}
Here we see essentially the same phenomenon as we described for circuits above:
the apex of the cospan is much larger than the image of the maps from the feet.

Corelations address this with what is known as a $(\mc E,\mc M)$-factorisation
system. A factorisation system comprises subcategories $\mc E$ and $\mc M$ of
$\mc C$ such that every morphism in $\mc C$ factors, in a coherent way, as the
composite of a morphism in $\mc E$ followed by a morphism in $\mc M$. An
example, known as the epi-mono factorisation system on $\Set$, is yielded by the
observation that every function can be written as a surjection followed by an
injection.

Corelations, or more precisely $(\mc E,\mc M)$-corelations, are cospans $X
\to N \leftarrow Y$ such that the copairing $X+Y \to N$ of the two maps is an
element of the first factor $\mc E$ of the factorisation system. Composition of
corelations proceeds first as composition of cospans, but then takes only the
so-called $\mc E$-part of the composite cospan, to ensure the composite is again
a corelation. If we take the $\mc E$-part of a cospan $X \to N \leftarrow Y$, we
write the new apex $\overline{N}$, and so the resulting corelation $X \to
\overline{N} \leftarrow Y$. 

Mapping the above two cospans to epi-mono corelations in $\FinSet$ they become 
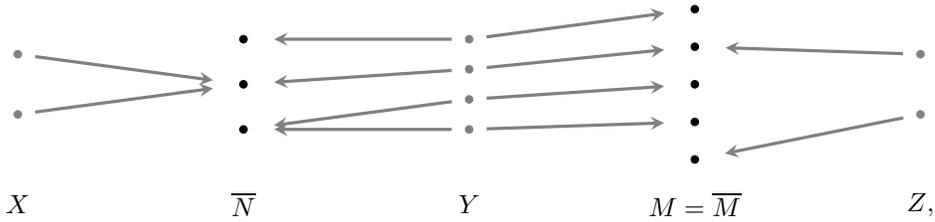
\begin{center}
  \begin{tikzpicture}[auto,scale=2]
    \node[circle,draw,inner sep=1pt,fill=gray,color=gray] (x1) at (-1.5,.2) {};
    \node[circle,draw,inner sep=1pt,fill=gray,color=gray] (x2) at (-1.5,-.2) {};
    \node at (-1.5,-.8) {$X$};
    \node[circle,draw,inner sep=1pt,fill]         (B) at (0,.3) {};
    \node[circle,draw,inner sep=1pt,fill]         (C) at (0,0) {};
    \node[circle,draw,inner sep=1pt,fill]         (D) at (0,-.3) {};
    \node at (0,-.8) {$\overline{N}$};
    \node[circle,draw,inner sep=1pt,fill=gray,color=gray] (y1) at (1.5,.3) {};
    \node[circle,draw,inner sep=1pt,fill=gray,color=gray] (y2) at (1.5,.1) {};
    \node[circle,draw,inner sep=1pt,fill=gray,color=gray] (y3) at (1.5,-.1) {};
    \node[circle,draw,inner sep=1pt,fill=gray,color=gray] (y4) at (1.5,-.3) {};
    \node at (1.5,-.8) {$Y$};
    \node[circle,draw,inner sep=1pt,fill]         (A') at (3,.5) {};
    \node[circle,draw,inner sep=1pt,fill]         (B') at (3,.25) {};
    \node[circle,draw,inner sep=1pt,fill]         (C') at (3,0) {};
    \node[circle,draw,inner sep=1pt,fill]         (D') at (3,-.25) {};
    \node[circle,draw,inner sep=1pt,fill]         (E') at (3,-.5) {};
    \node at (3,-.8) {$M=\overline{M}$};
    \node[circle,draw,inner sep=1pt,fill=gray,color=gray] (z1) at (4.5,.2) {};
    \node[circle,draw,inner sep=1pt,fill=gray,color=gray] (z2) at (4.5,-.2) {};
    \node at (4.5,-.8) {$Z$,};
    \path[color=gray, very thick, shorten >=10pt, shorten <=5pt, ->, >=stealth]
    (x1) edge (C);
    \path[color=gray, very thick, shorten >=10pt, shorten <=5pt, ->, >=stealth]
    (x2) edge (C);
    \path[color=gray, very thick, shorten >=10pt, shorten <=5pt, ->, >=stealth]
    (y1) edge (B);
    \path[color=gray, very thick, shorten >=10pt, shorten <=5pt, ->, >=stealth]
    (y2) edge (C);
    \path[color=gray, very thick, shorten >=10pt, shorten <=5pt, ->, >=stealth]
    (y3) edge (D);
    \path[color=gray, very thick, shorten >=10pt, shorten <=5pt, ->, >=stealth]
    (y4) edge (D);
    \path[color=gray, very thick, shorten >=10pt, shorten <=5pt, ->, >=stealth]
    (y1) edge (A');
    \path[color=gray, very thick, shorten >=10pt, shorten <=5pt, ->, >=stealth]
    (y2) edge (B');
    \path[color=gray, very thick, shorten >=10pt, shorten <=5pt, ->, >=stealth]
    (y3) edge (C');
    \path[color=gray, very thick, shorten >=10pt, shorten <=5pt, ->, >=stealth]
    (y4) edge (D');
    \path[color=gray, very thick, shorten >=10pt, shorten <=5pt, ->, >=stealth]
    (z1) edge (B');
    \path[color=gray, very thick, shorten >=10pt, shorten <=5pt, ->, >=stealth]
    (z2) edge (E');
  \end{tikzpicture}
\end{center}
with composite
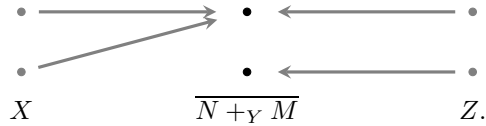
\begin{center}
  \begin{tikzpicture}[auto,scale=2]
    \node[circle,draw,inner sep=1pt,fill=gray,color=gray] (x1) at (1.5,.2) {};
    \node[circle,draw,inner sep=1pt,fill=gray,color=gray] (x2) at (1.5,-.2) {};
    \node at (1.5,-.45) {$X$};
    \node[circle,draw,inner sep=1pt,fill]         (A) at (3,.2) {};
    \node[circle,draw,inner sep=1pt,fill]         (B) at (3,-.2) {};
    \node at (3,-.45) {$\overline{N+_YM}$};
    \node[circle,draw,inner sep=1pt,fill=gray,color=gray] (z1) at (4.5,.2) {};
    \node[circle,draw,inner sep=1pt,fill=gray,color=gray] (z2) at (4.5,-.2) {};
    \node at (4.5,-.45) {$Z$.};
    \path[color=gray, very thick, shorten >=10pt, shorten <=5pt, ->, >=stealth]
    (x1) edge (A);
    \path[color=gray, very thick, shorten >=10pt, shorten <=5pt, ->, >=stealth]
    (x2) edge (A);
    \path[color=gray, very thick, shorten >=10pt, shorten <=5pt, ->, >=stealth]
    (z1) edge (A);
    \path[color=gray, very thick, shorten >=10pt, shorten <=5pt, ->, >=stealth]
    (z2) edge (B);
  \end{tikzpicture}
\end{center}
Note that the apex of the composite corelation is the subset of the apex of the
composite cospan comprising exactly those elements in the image of the maps from
the feet. The intuition, again, is that composition of corelations discards
irrelevant information---of course, exactly what information it discards depends
on our choice of factorisation system.

Recall that a hypergraph category is a symmetric monoidal category in which
every object is equipped with the structure of a special commutative Frobenius
monoid in a coherent way. Due to the readily available Frobenius structure,
hypergraph categories are well suited to modelling network-style composition.
For example, in the circuits example above, the Frobenius structure allows
description of many-to-many interconnections, as well as the ability to turn
inputs into outputs, and vice versa. 

Our first contribution is to show that, under a mild condition on the
factorisation system, we can use corelations to construct
hypergraph categories and their functors.

\begin{theorem}
Let $\mc C$ be a category with finite colimits and a factorisation system $(\mc
E,\mc M)$. If $\mc M$ is stable under pushout, then corelations in $\mc C$ form
the morphisms of a hypergraph category. 

Moreover, let $A$ be a colimit-preserving functor between categories $\mathcal
C$, $\mathcal C'$, where $\mc C$ and $\mc C'$ are respectively equipped with
factorisation systems $(\mathcal E, \mathcal M)$, $(\mathcal E', \mathcal M')$
such that $\mc M$ and $\mc M'$ are stable under pushout. If the image under $A$
of $\mathcal M$ lies in $\mathcal M'$, then $A$ induces a hypergraph functor
between their corelation categories.
\end{theorem}

The object of this paper is to construct hypergraph categories of
\emph{decorated} corelations. How do decorations enter the picture? An
instructive example comes from matrices. Suppose we have devices built from
channels that take the signal at some input, amplify it, and deliver it to some
output. For simplicity let these signals be real numbers, and amplication be
linear: we just multiply by some fixed scalar.  We depict an example device like
so:
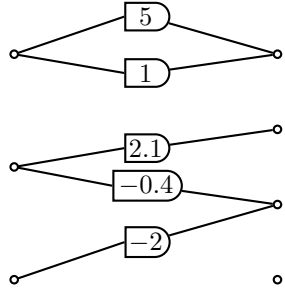
\[
    \tikzset{every path/.style={line width=.8pt}}
\begin{tikzpicture}
	\begin{pgfonlayer}{nodelayer}
		\node [style=sdot] (0) at (-2.5, 1.5) {};
		\node [style=sdot] (1) at (-2.5, -0) {};
		\node [style=sdot] (2) at (-2.5, -1.5) {};
		\node [style=amp] (3) at (-0.75, 2) {$5$};
		\node [style=amp] (4) at (-0.75, 1.25) {$1$};
		\node [style=amp] (5) at (-0.75, -1) {$-2$};
		\node [style=amp] (6) at (-0.75, 0.25) {$2.1$};
		\node [style=amp] (7) at (-0.75, -0.25) {$-0.4$};
		\node [style=sdot] (8) at (1, -0.5) {};
		\node [style=sdot] (9) at (1, -1.5) {};
		\node [style=sdot] (10) at (1, 0.5) {};
		\node [style=sdot] (11) at (1, 1.5) {};
	\end{pgfonlayer}
	\begin{pgfonlayer}{edgelayer}
		\draw (3.west) to (0.center);
		\draw (4.west) to (0.center);
		\draw (6.west) to (1.center);
		\draw (7.west) to (1.center);
		\draw (5.west) to (2.center);
		\draw (3) to (11);
		\draw (4) to (11);
		\draw (6) to (10);
		\draw (7) to (8);
		\draw (5) to (8);
	\end{pgfonlayer}
\end{tikzpicture}
\]
Here there are three inputs, four outputs, and five paths. Formally, we might
model these devices as finite sets of inputs $X$, outputs $Y$, and paths $N$,
together with functions $i\maps N \to X$ and $o\maps N \to Y$ describing the
start and end of each path, and a function $s\maps N \to \R$ describing the
amplification along it. In other words, these are cospans $X \to N \leftarrow Y$
in $\FinSet^\opp$ decorated by `scalar assignment' functions $N \to \R$. This
suggests a decorated cospans construction.

For a decorated cospan category we begin with a lax symmetric monoidal functor
on a category with finite colimits, such as the functor that takes a finite set
$N$ to the set of circuits with vertex set $N$. Here we begin with the lax
symmetric monoidal functor $\mathbb R^{(-)}\maps$ $\FinSet^\opp \to \Set$ that
takes a finite set $N$ to the set $\R^N$ of functions $s\maps N \to \R$, and
takes an opposite function $f^\opp\maps N \to M$ to the map sending $s\maps N
\to \R$ to $s \circ f \maps M \to \R$. Note that the coproduct in $\FinSet^\opp$
is the cartesian product of sets. The coherence maps of the functor, which are
critical for composing the decorations, are given by $\varphi_{N,M}\maps \R^N
\times \R^M \to \R^{N\times M}$, taking $(s,t) \in \R^N \times \R^M$ to the
function $s \cdot t\maps N \times M \to \R$ defined by pointwise multiplication
in $\R$. 

Composition in this decorated cospan category is thus given by the
multiplication in $\R$. In detail, given decorated cospans $(X
\xrightarrow{i_X^\opp} N \xleftarrow{o_Y^\opp} Y, N \xrightarrow{s} \R)$ and $(Y
\xrightarrow{i_Y^\opp} M \xleftarrow{o_Z^\opp} Z, M \xrightarrow{t} \R)$, the
composite has a path from $x \in X$ to $z \in Z$ for every triple $(y,n,m)$
where $y \in Y$, $n \in N$, and $m \in M$, such that $n$ is a path from $x$ to
$y$ and $m$ is a path from $y$ to $z$. The scalar assigned to this path is the
product of those assigned to $n$ and $m$. For example, we have the following
composite
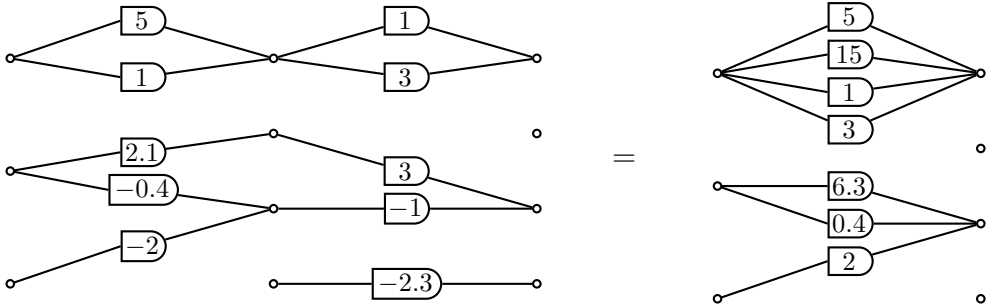
\[
    \tikzset{every path/.style={line width=.8pt}}
    \begin{aligned}
\begin{tikzpicture}
	\begin{pgfonlayer}{nodelayer}
		\node [style=sdot] (0) at (-2.5, 1.5) {};
		\node [style=sdot] (1) at (-2.5, -0) {};
		\node [style=sdot] (2) at (-2.5, -1.5) {};
		\node [style=amp] (3) at (-0.75, 2) {$5$};
		\node [style=amp] (4) at (-0.75, 1.25) {$1$};
		\node [style=amp] (5) at (-0.75, -1) {$-2$};
		\node [style=amp] (6) at (-0.75, 0.25) {$2.1$};
		\node [style=amp] (7) at (-0.75, -0.25) {$-0.4$};
		\node [style=sdot] (8) at (1, -0.5) {};
		\node [style=sdot] (9) at (4.5, 0.5) {};
		\node [style=sdot] (10) at (1, -1.5) {};
		\node [style=sdot] (11) at (4.5, -1.5) {};
		\node [style=amp] (12) at (2.75, -0.5) {$-1$};
		\node [style=amp] (13) at (2.75, 1.25) {$3$};
		\node [style=amp] (14) at (2.75, -1.5) {$-2.3$};
		\node [style=amp] (15) at (2.75, 2) {$1$};
		\node [style=sdot] (16) at (4.5, 1.5) {};
		\node [style=amp] (17) at (2.75, -0) {$3$};
		\node [style=sdot] (18) at (4.5, -0.5) {};
		\node [style=sdot] (19) at (1, 0.5) {};
		\node [style=sdot] (20) at (1, 1.5) {};
	\end{pgfonlayer}
	\begin{pgfonlayer}{edgelayer}
		\draw (3.west) to (0.center);
		\draw (4.west) to (0.center);
		\draw (6.west) to (1.center);
		\draw (7.west) to (1.center);
		\draw (5.west) to (2.center);
		\draw (16) to (15);
		\draw (16) to (13);
		\draw (18) to (12);
		\draw (14) to (11);
		\draw (10) to (14);
		\draw (15) to (20);
		\draw (13) to (20);
		\draw (3) to (20);
		\draw (4) to (20);
		\draw (6) to (19);
		\draw (7) to (8);
		\draw (5) to (8);
		\draw (12) to (8);
		\draw (19) to (17);
		\draw (17) to (18);
	\end{pgfonlayer}
\end{tikzpicture}
\end{aligned}
\qquad = \qquad
\begin{aligned}
\begin{tikzpicture}
	\begin{pgfonlayer}{nodelayer}
		\node [style=sdot] (0) at (-2.5, 1.5) {};
		\node [style=sdot] (1) at (-2.5, -0) {};
		\node [style=sdot] (2) at (-2.5, -1.5) {};
		\node [style=amp] (3) at (-0.75, 1.75) {$15$};
		\node [style=amp] (4) at (-0.75, 1.25) {$1$};
		\node [style=amp] (5) at (-0.75, -1) {$2$};
		\node [style=amp] (6) at (-0.75, -0) {$6.3$};
		\node [style=amp] (7) at (-0.75, -0.5) {$0.4$};
		\node [style=sdot] (8) at (1, -0.5) {};
		\node [style=sdot] (9) at (1, -1.5) {};
		\node [style=sdot] (10) at (1, 0.5) {};
		\node [style=sdot] (11) at (1, 1.5) {};
		\node [style=amp] (12) at (-0.75, 2.25) {$5$};
		\node [style=amp] (13) at (-0.75, 0.75) {$3$};
	\end{pgfonlayer}
	\begin{pgfonlayer}{edgelayer}
		\draw (3.west) to (0.center);
		\draw (4.west) to (0.center);
		\draw (6.west) to (1.center);
		\draw (7.west) to (1.center);
		\draw (5.west) to (2.center);
		\draw (3) to (11);
		\draw (4) to (11);
		\draw (7) to (8);
		\draw (5) to (8);
		\draw (0) to (12);
		\draw (12) to (11);
		\draw (0) to (13);
		\draw (13) to (11);
		\draw (6) to (8);
	\end{pgfonlayer}
\end{tikzpicture}
\end{aligned}
\]
There are four paths between the top-most element $x_1$ of the domain and the
top-most element $z_1$ of the codomain: we may first take the path that
amplifies by $5\times$ and then the path that amplifies by $1\times$ for a total
amplification of $5\times$, or $5\times$ and $3\times$ for $15\times$, and so
on. This means we end up with four elements relating $x_1$ and $z_1$ in the
composite. The apex of the composite is in fact given by the pullback $N
\times_Y M$ of the cospan $N \xrightarrow{o_Y} Y \xleftarrow{i_Y} M$ in
$\FinSet$.

Here we again see the problem of decorated cospans: the composite of the above
puts decorations on $N\times_Y M$, which can be of much larger cardinality than
$N$ and $M$. We wish to avoid the size of our decorated cospan from growing so
fast. Moreover, from our open systems perspective, we care not about the path
but by the total amplification of the signal from some chosen input to some
chosen output. The intuition is that if we black box the system, then we cannot
tell what paths the signal took through the system, only the total amplification
from input to output.

We thus want to restrict our apex to contain at most one point for each
input--output pair $(x,y)$. We do this by pushing the decoration along the
surjection $e$ in the epi-mono factorisation of the function $N\times_YM
\xrightarrow{e} \overline{N\times_YM} \xrightarrow{m} X \times Z$. Put another
way, we want the category of decorated corelations, not decorated cospans.

Represented as decorated corelations, the above composite becomes
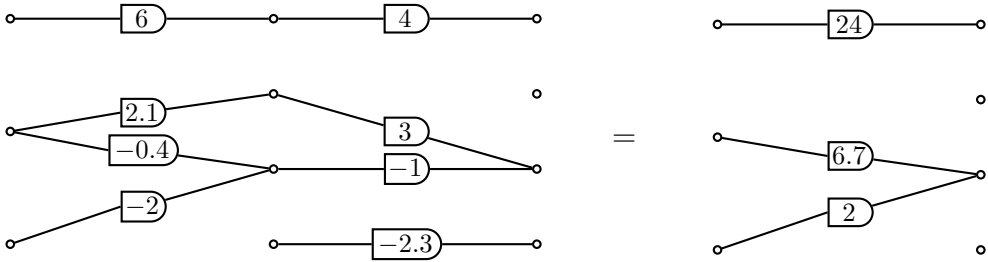
\[
  \tikzset{every path/.style={line width=.8pt}}
  \begin{aligned}
    \begin{tikzpicture}
      \begin{pgfonlayer}{nodelayer}
	\node [style=sdot] (0) at (-2.5, 1.5) {};
	\node [style=sdot] (1) at (-2.5, -0) {};
	\node [style=sdot] (2) at (-2.5, -1.5) {};
	\node [style=amp] (3) at (-0.75, 1.5) {$6$};
	\node [style=amp] (4) at (-0.75, -1) {$-2$};
	\node [style=amp] (5) at (-0.75, 0.25) {$2.1$};
	\node [style=amp] (6) at (-0.75, -0.25) {$-0.4$};
	\node [style=sdot] (7) at (1, -0.5) {};
	\node [style=sdot] (8) at (4.5, 0.5) {};
	\node [style=sdot] (9) at (1, -1.5) {};
	\node [style=sdot] (10) at (4.5, -1.5) {};
	\node [style=amp] (11) at (2.75, -0.5) {$-1$};
	\node [style=amp] (12) at (2.75, -1.5) {$-2.3$};
	\node [style=amp] (13) at (2.75, 1.5) {$4$};
	\node [style=sdot] (14) at (4.5, 1.5) {};
	\node [style=amp] (15) at (2.75, -0) {$3$};
	\node [style=sdot] (16) at (4.5, -0.5) {};
	\node [style=sdot] (17) at (1, 0.5) {};
	\node [style=sdot] (18) at (1, 1.5) {};
      \end{pgfonlayer}
      \begin{pgfonlayer}{edgelayer}
	\draw (3.west) to (0.center);
	\draw (5.west) to (1.center);
	\draw (6.west) to (1.center);
	\draw (4.west) to (2.center);
	\draw (14) to (13);
	\draw (16) to (11);
	\draw (12) to (10);
	\draw (9) to (12);
	\draw (13) to (18);
	\draw (3) to (18);
	\draw (5) to (17);
	\draw (6) to (7);
	\draw (4) to (7);
	\draw (11) to (7);
	\draw (17) to (15);
	\draw (15) to (16);
      \end{pgfonlayer}
    \end{tikzpicture}
  \end{aligned}
  \qquad = \qquad
  \begin{aligned}
    \begin{tikzpicture}
      \begin{pgfonlayer}{nodelayer}
	\node [style=sdot] (0) at (-2.5, 1.5) {};
	\node [style=sdot] (1) at (-2.5, -0) {};
	\node [style=sdot] (2) at (-2.5, -1.5) {};
	\node [style=amp] (3) at (-0.75, 1.5) {$24$};
	\node [style=amp] (4) at (-0.75, -1) {$2$};
	\node [style=amp] (5) at (-0.75, -0.25) {$6.7$};
	\node [style=sdot] (6) at (1, -0.5) {};
	\node [style=sdot] (7) at (1, -1.5) {};
	\node [style=sdot] (8) at (1, 0.5) {};
	\node [style=sdot] (9) at (1, 1.5) {};
      \end{pgfonlayer}
      \begin{pgfonlayer}{edgelayer}
	\draw (3.west) to (0.center);
	\draw (5.west) to (1.center);
	\draw (4.west) to (2.center);
	\draw (3) to (9);
	\draw (5) to (6);
	\draw (4) to (6);
      \end{pgfonlayer}
    \end{tikzpicture}
  \end{aligned}
\]
Note that composite is not simply the composite as decorated cospans, but the
composite decorated cospan reduced to a decorated corelation. We will show that
this decorated corelations category is equivalent to the category of real vector
spaces and linear maps, with monoidal product the tensor product.

It is not a trivial fact that the above composition rule for decorated
corelations defines a category. Indeed, the reason that it is possible to push
the decoration along the surjection $e$ is that the lax symmetric monoidal
functor $\R^{(-)}\maps \FinSet^\opp \to \Set$ extends to a lax symmetric
monoidal functor $(\FinSet^\opp;\mathrm{Sur}^\opp) \to \Set$. Here
$\FinSet^\opp;\mathrm{Sur}^\opp$ is the subcategory of $\cospan(\FinSet^\opp)$
comprising cospans of the form $\xrightarrow{f^\opp}\xleftarrow{e^\opp}$, where
$f$ is any function but $e$ must be a surjection.

More generally, given a category $\mc C$ with finite colimits and a subcategory
$\mc M$ stable under pushouts, we may construct a symmetric monoidal category
$\mc C;\mc M^\opp$ with isomorphisms classes of cospans of the form
$\xrightarrow{f}\xleftarrow{m}$, where $f \in \mc C$, $m \in \mc M$, as
morphisms. The monoidal product is again derived from the coproduct in $\mc C$.

The main theorem is that these decorated corelations form a hypergraph category.
\begin{theorem}
Given a category $\mc C$ with finite colimits, factorisation system $(\mc E,\mc
M)$ such that $\mc M$ is stable under pushouts, and a symmetric lax monoidal functor 
\[
  F\maps \mc C; \mc M^\opp \longrightarrow \Set,
\]
define a decorated corelation to be an $(\mc E,\mc
M)$-corelation $X \to N \leftarrow Y$ in $\mc C$ together with an element of
$FN$. Then there is a hypergraph category $F\mathrm{Corel}$ with the objects of
$\mc C$ as objects and isomorphism classes of decorated corelations as
morphisms. 
\end{theorem}

As for decorated cospans, hypergraph functors between these so-named decorated
corelations categories can further be defined from natural transformations
between the decorating functors. This is especially useful for problems of
constructing compositional semantics, such as the circuit setting outlined
above.

\subsection*{Outline.}
The structure of this paper is straightforward. After a brief review of
background material, we discuss in turn corelation categories
(\textsection\ref{sec.corels}), functors between corelation categories
(\textsection\ref{sec.corelfunctors}), decorated corelation categories
(\textsection\ref{sec.dcorc}), and functors between decorated corelation
categories (\textsection\ref{sec.dcorf}). We then conclude with detailed
discussions of two examples: matrices and linear relations.

\section{Background} \label{sec.background}

This section provides a brief review of hypergraph categories, cospans,
decorated cospans, and corelations. For details, see \cite{Fon15, Fon16}.

\subsection*{Hypergraph categories}\hfill

\noindent
We recall special commutative Frobenius monoids, writing
our axioms using the string calculus for monoidal categories introduced by Joyal
and Street \cite{JS91}. Diagrams will be read left to right, and we shall suppress
the labels as we deal with a unique generating object and a unique generator of
each type. 

\begin{definition}
  A \define{special commutative Frobenius monoid} $(X,\mu,\eta,\delta,\epsilon)$
  in a monoidal category $(\mathcal C, \otimes)$ is an object $X$ of $\mathcal
  C$ together with maps 
\[
  \xymatrixrowsep{1pt}
  \xymatrix{
    \mult{.075\textwidth} & & \unit{.075\textwidth} & & 
    \comult{.075\textwidth} & & \counit{.075\textwidth} \\
    \mu\maps X\otimes X \to X & & \eta\maps I \to X & & 
    \delta\maps X\to X \otimes X & & \epsilon\maps X \to I
  }
\]
obeying the commutative monoid axioms
\[
  \xymatrixrowsep{1pt}
  \xymatrixcolsep{25pt}
  \xymatrix{
    \assocl{.1\textwidth} = \assocr{.1\textwidth} & \unitl{.1\textwidth} =
    \idone{.1\textwidth} & \commute{.1\textwidth} = \mult{.07\textwidth} \\
    \textrm{(associativity)} & \textrm{(unitality)} & \textrm{(commutativity)}
  }
\]
the cocommutative comonoid axioms
\[
  \xymatrixrowsep{1pt}
  \xymatrixcolsep{25pt}
  \xymatrix{
    \coassocl{.1\textwidth} = \coassocr{.1\textwidth} & \counitl{.1\textwidth} =
    \idone{.1\textwidth} & \cocommute{.1\textwidth} = \comult{.07\textwidth} \\
    \textrm{(coassociativity)} & \textrm{(counitality)} &
    \textrm{(cocommutativity)}
  }
\]
and the Frobenius and special axioms
  \[
  \xymatrixrowsep{1pt}
  \xymatrixcolsep{25pt}
  \xymatrix{
    \frobs{.1\textwidth} = \frobx{.1\textwidth} = \frobz{.1\textwidth} & \spec{.1\textwidth} =
    \idone{.1\textwidth} \\
    \textrm{(Frobenius)} & \textrm{(special)} 
  }
  \]
where $\swap{1em}$ is the braiding on $X \otimes X$.   
\end{definition}

\begin{definition}
  A \define{hypergraph category} is a symmetric monoidal category in which each
  object $X$ is equipped with a special commutative Frobenius structure
  $(X,\mu_X,\delta_X,\eta_X,\epsilon_X)$ such that 
\[
  \begin{array}{cc}
    \mu_{X\otimes Y} = (\mu_X \otimes \mu_Y)\circ(1_X \otimes \sigma_{YX}\otimes
    1_Y) \qquad&
    \eta_{X\otimes Y} = \eta_X \otimes \eta_Y \\
    \delta_{X\otimes Y} = (1_X \otimes \sigma_{XY}\otimes 1_Y)\circ(\delta_X
    \otimes \delta_Y) \qquad&
    \epsilon_{X\otimes Y} = \epsilon_X \otimes \epsilon_Y.
  \end{array}
\]
A functor $(F,\varphi)$ of hypergraph categories, or \define{hypergraph
functor}, is a strong symmetric monoidal functor $(F,\varphi)$ that preserves
the hypergraph structure. More precisely, the latter condition means that given
an object $X$, the special commutative Frobenius structure on $FX$ must be 
\[
  (FX,\enspace F\mu_X \circ \varphi_{X,X},\enspace  \varphi^{-1} \circ F\delta_X,\enspace  F\eta_X \circ
\varphi_1,\enspace  \varphi_1 \circ \epsilon_X).
\]
\end{definition}

Hypergraph categories were first defined by Carboni and Walters, under the name
well-supported compact closed categories \cite{Car91}.
\smallskip

\subsection*{Cospans}\hfill

\noindent
We give a fundamental example of hypergraph categories.

Let $\mc C$ be a category with finite colimits.
Recall that a \define{cospan} $X \stackrel{i}{\longrightarrow} N
\stackrel{o}{\longleftarrow} Y$  from $X$ to $Y$ in $\mathcal C$ is a pair of
morphisms with common codomain. We refer to $X$ and $Y$ as the \define{feet},
and $N$ as the \define{apex}.  Given two cospans $X
\stackrel{i}{\longrightarrow} N \stackrel{o}{\longleftarrow} Y$ and $X
\stackrel{i'}{\longrightarrow} N' \stackrel{o'}{\longleftarrow} Y$ with the same
feet, a \define{map of cospans} is a morphism $n\colon  N \to N'$ in $\mathcal
C$ between the apices such that
\[
  \xymatrix{
    & N \ar[dd]^n  \\
    X \ar[ur]^{i} \ar[dr]_{i'} && Y \ar[ul]_{o} \ar[dl]^{o'}\\
    & N'
  }
\]
commutes.

Cospans may be composed, up to isomorphism, using the pushout from the common
foot: given cospans $X \stackrel{i_X}{\longrightarrow} N
\stackrel{o_Y}{\longleftarrow} Y$ and $Y \stackrel{i_Y}{\longrightarrow} M
\stackrel{o_Z}{\longleftarrow} Z$, their composite cospan is $X \xrightarrow{j_N
  \circ i_X} N+_YM \xleftarrow{j_M\circ i_Z} Z$,
  where 
\[
  \xymatrix{
    && N+_YM \\
    & N \ar[ur]^{j_N} && M \ar[ul]_{j_M} \\
    \quad X \quad \ar[ur]^{i_X} && Y \ar[ul]_{o_Y} \ar[ur]^{i_Y} && \quad Z \quad \ar[ul]_{o_Z}
  }
\]
is a pushout square. 

Write $+$ for the coproduct in $\mc C$. We may consider $\mc C$ as a symmetric
monoidal category $(\mc C,+)$ with monoidal product given by the coproduct.
Also, given maps $f \maps A \to C$, $g \maps B \to C$ with common codomain, the
universal property of the coproduct gives a unique map $[f,g]\maps A+B \to C$;
we call this the \define{copairing} of $f$ and $g$. We write $!\maps \varnothing
\to X$ for the unique map from the initial object.

\begin{proposition}[{\cite[\textsection 2.2]{RSW08}}]
  Given a category $\mc C$ with finite colimits, we may define a hypergraph
  category $\cospan(\mc C)$ as follows:
\smallskip 

  \begin{center}
  \begin{tabular}{ |c| p{.65\textwidth}|}
      \hline
      \multicolumn{2}{|c|}{The hypergraph category $(\mathrm{Cospan(\mc C)},+)$} \\
    \hline
    \textbf{objects} & the objects of $\mathcal C$ \\ 
    \textbf{morphisms} & isomorphism classes of cospans in
    $\mathcal C$\\ 
  \textbf{composition} & given by pushout \\
  \textbf{monoidal product} & the coproduct in $\mathcal C$ \\
  \textbf{coherence maps} & inherited from $(\mc C,+)$\\
  \textbf{hypergraph maps} & $\mu = [1,1]$, $\eta = !$,
  $\delta = [1,1]^\opp$, $\epsilon = !^\opp$. \\
      \hline
  \end{tabular}
\end{center}
\smallskip
\end{proposition}

Note that, as we do above, we shall frequently use representative cospans to
refer to their isomorphism class.  Furthermore, given $f \maps X \to Y$ in
$\mc C$, we also abuse notation by writing $f \in \mathrm{Cospan}(\mc C)$ for
the cospan $X \stackrel{f}\to Y \stackrel{1_Y}\leftarrow Y$, and $f^\opp$ for
the cospan $Y \stackrel{1_Y}\to Y \stackrel{f}\leftarrow X$.
\smallskip

\subsection*{Decorated cospans}\hfill

\noindent
Write $(\Set,\times)$ for the symmetric monoidal category of finite sets and
functions, where the monoidal product is the categorical product.

\begin{definition} \label{def:fcospans}
  Let $\mathcal C$ be a category with finite colimits, and
  \[
    (F,\varphi): (\mathcal C,+) \longrightarrow (\Set, \times)
  \]
  be a symmetric lax monoidal functor. We define a \define{decorated cospan}, or more
  precisely an $F$-decorated cospan, to be a pair 
  \[
    \left(
    \begin{aligned}
      \xymatrix{
	& N \\  
	X \ar[ur]^{i} && Y \ar[ul]_{o}
      }
    \end{aligned}
    ,
    \qquad
    \begin{aligned}
      \xymatrix{
	FN \\
	1 \ar[u]_{s}
      }
    \end{aligned}
    \right)
  \]
  comprising a cospan $X \stackrel{i}\rightarrow N \stackrel{o}\leftarrow Y$ in
  $\mathcal C$ together with an element $1 \stackrel{s}\rightarrow FN$ of
  the $F$-image $FN$ of the apex of the cospan. The element $1
  \stackrel{s}\rightarrow FN$ is known as the \define{decoration} of the decorated
  cospan. A morphism of decorated cospans 
  \[
    n: \big(X \stackrel{i_X}\longrightarrow N \stackrel{o_Y}\longleftarrow
    Y,\enspace 1 \stackrel{s}\longrightarrow FN\big) \longrightarrow \big(X
    \stackrel{i'_X}\longrightarrow N' \stackrel{o'_Y}\longleftarrow Y,\enspace 1
    \stackrel{s'}\longrightarrow FN'\big)
  \]
  is a morphism $n: N \to N'$ of cospans such that $Fn \circ s = s'$.
\end{definition}

On representatives of the isomorphism classes, composition of decorated
cospans is given by the usual composite cospan decorated with the composite
\[
  1 \stackrel{\lambda^{-1}}\longrightarrow 1 \otimes 1 \stackrel{s \otimes
  t}\longrightarrow FN \otimes FM \stackrel{\varphi_{N,M}}\longrightarrow
  F(N+M) \stackrel{F[j_N,j_M]}\longrightarrow F(N+_YM)
\]
of the tensor product of the decorations with the $F$-image of the copairing
of the pushout maps.

Each cospan may be given the empty decoration. The \define{empty decoration} is the
decoration $1 \stackrel{\varphi_1}{\rightarrow} F\varnothing
\stackrel{F!}{\rightarrow} FN$.

\begin{proposition}[{\cite[Theorem 3.4]{Fon15}}] \label{thm:fcospans}
  Let $\mathcal C$ be a category with finite colimits and $(F,\varphi):
  (\mathcal C,+) \to (\Set, \times)$ a symmetric lax monoidal functor. We
  define:
\begin{center}
  \begin{tabular}{| c | p{.65\textwidth} |}
    \hline
    \multicolumn{2}{|c|}{The hypergraph category $(F\mathrm{Cospan},+)$} \\
    \hline
    \textbf{objects} & the objects of $\mathcal C$ \\ 
    \textbf{morphisms} & isomorphism classes of $F$-decorated cospans in
    $\mathcal C$\\ 
    \textbf{composition} & given by pushout, as described above \\
    \textbf{monoidal product} & the coproduct in $\mathcal C$ \\
    \textbf{coherence maps} &  maps from $\cospan(\mc C)$ with empty decoration \\
    \textbf{hypergraph maps} & maps from $\cospan(\mc C)$ with empty decoration
    \\
    \hline
  \end{tabular}
\end{center}
\end{proposition}

\begin{remark}
In previous expositions of decorated cospans we have let decorations lie in any
braided monoidal category. Sam Staton pointed out that it is general enough to
let decorations lie in the symmetric monoidal category $(\Set,\times)$. See
Appendix \ref{ssec.setdecorations} for details.
\end{remark}

We will also need the following lemma, showing that empty decorations act
trivially.

\begin{lemma}[{\cite[Proposition A.4]{Fon15}}] \label{lem.emptydecorations}
  Let $(X \stackrel{i_X}\longrightarrow N
\stackrel{o_Y}\longleftarrow Y,\enspace I \stackrel{s}\longrightarrow FN)$
be a decorated cospan, and suppose we have an empty-decorated cospan $(Y \stackrel{i_Y}\longrightarrow M \stackrel{o_Z}\longleftarrow Z,\enspace I
  \stackrel{\varphi\circ F!}\longrightarrow FM)$.
Then the composite of these decorated cospans is 
\[
  \big(X \xrightarrow{j_N \circ i_X} N+_YM \xleftarrow{j_M \circ o_Z} Z,\enspace 
  I \xrightarrow{Fj_N \circ s} F(N+_YM)\big).
\]
In particular, the decoration on the composite is the decoration $s$ pushed
forward along the $F$-image of the map $j_N\colon N \to N+_YM$ to become an
$F$-decoration on $N+_YM$. The analogous statement also holds for composition
with an empty-decorated cospan on the left.
\end{lemma}
\smallskip

\subsection*{Corelations}\hfill

\noindent
Given sets $X$, $Y$, a relation $X \to Y$ is a subset of the product $X
\times Y$. Note that by the universal property of the product, spans $X
\leftarrow N \to Y$ are in one-to-one correspondence with functions $N \to X
\times Y$. When this map is monic, we say that the span is \emph{jointly monic}.
More abstractly then, we might say a relation is an isomorphism class of jointly
monic spans in the category of sets. Here we generalise the dual concept: these
are our so-called corelations.

The category theoretic study of relations is extensive; for a survey, see
\cite{Mil00}.  In our general setting, the key insight is the use of a
factorisation system. A factorisation system allows any morphism in a category
to be factored into the composite of two morphisms in a coherent way.

\begin{definition}
  A \define{factorisation system} $(\mathcal E,\mathcal M)$ in a category
  $\mathcal C$ comprises subcategories $\mathcal E$, $\mathcal M$ of $\mathcal
  C$ such that
  \begin{enumerate}[(i)]
    \item $\mathcal E$ and $\mathcal M$ contain all isomorphisms of $\mathcal
      C$.
    \item  every morphism $f \in \mathcal C$ admits a factorisation $f=m \circ
      e$, $e \in \mathcal E$, $m \in \mathcal M$.
\item given morphisms $f,f'$, with factorisations $f = m \circ e$, $f' = m' \circ
  e'$ of the above sort, for every $u$, $v$ such that $v \circ f = f' \circ u$, there exists a unique morphism $s$ such that
  \[
    \xymatrixcolsep{3pc}
    \xymatrixrowsep{3pc}
    \xymatrix{
      \ar[r]^e \ar[d]_u & \ar[r]^m \ar@{.>}[d]^{\exists! s} &  \ar[d]^v \\
       \ar[r]_{e'}& \ar[r]_{m'} & 
    }
  \]
  commutes.
  \end{enumerate}
\end{definition}

Observe that relations are just spans $X \leftarrow N \to Y$ in $\Set$ such that
$N \to X \times Y$ is an element of $\mathrm{Inj}$, the right factor in the
factorisation system $(\mathrm{Sur},\mathrm{Inj})$. Relations may thus be
generalised as spans such that the span maps jointly belong to some class $\mc
M$ of an $(\mc E,\mc M)$-factorisation system. We define corelations in the dual
manner.

\begin{definition}
  Let $\mathcal C$ be a category with finite colimits, and let $(\mathcal E,
  \mathcal M)$ be a factorisation system on $\mathcal C$. An $(\mathcal
  E,\mathcal M)$\define{-corelation} $X \to Y$ is a cospan $X
  \stackrel{i}\longrightarrow N \stackrel{o}\longleftarrow Y$ in $\mc C$ such
  that the copairing $[i,o]\maps X+Y \to N$ lies in $\mathcal E$.
\end{definition}

When the factorisation system is clear from context, we simply call $(\mathcal
E,\mathcal M)$-corelations `corelations'.

We also say that a cospan $X \stackrel{i}\longrightarrow N
\stackrel{o}\longleftarrow Y$ with the property that the copairing $[i,o]\maps
X+Y \to N$ lies in $\mathcal E$ is \define{jointly} $\mathcal E$\define{-like}.
Note that if a cospan is jointly $\mc E$-like then so are all isomorphic
cospans. Thus the property of being a corelation is closed under isomorphism of
cospans, and we again are often lazy with our language, referring to both
jointly $\mc E$-like cospans and their isomorphism classes as corelations. 

If $f\maps A \to N$ is a morphism with factorisation $f = m \circ e$, write
$\overline N$ for the object such that $e\maps A \to \overline N$ and $m\maps
\overline N \to N$. Now, given a cospan $X \stackrel{i_X}{\longrightarrow} N
\stackrel{o_Y}{\longleftarrow} Y$, we may use the factorisation system to write
the copairing $[i_X,o_Y]\maps X+Y \to N$ as
\[
  X+Y \stackrel{e}{\longrightarrow} \overline{N} \stackrel{m}{\longrightarrow}
  N.
\]
From the universal property of the coproduct, we also have maps $\iota_X\maps X
\to X+Y$ and $\iota_Y\maps Y \to X+Y$. We then call the corelation 
\[
  X \stackrel{e\circ \iota_X}{\longrightarrow} \overline{N} \stackrel{e \circ
  \iota_Y}{\longleftarrow} Y
\]
the $\mathcal E$\define{-part} of the above cospan. On occasion we will also
write $e\maps X+Y \to \overline N$ for the same corelation.

We compose corelations by taking the $\mathcal E$-part of their composite
cospan. That is, given corelations $X \stackrel{i_X}{\longrightarrow} N
\stackrel{o_Y}{\longleftarrow} Y$ and $Y \stackrel{i_Y}{\longrightarrow} M
\stackrel{o_Z}{\longleftarrow} Z$, their composite is given by the cospan $X
\xrightarrow{e\circ\iota_X} \overline{N+_YM} \xleftarrow{e \circ \iota_Z} Z$ in the commutative diagram
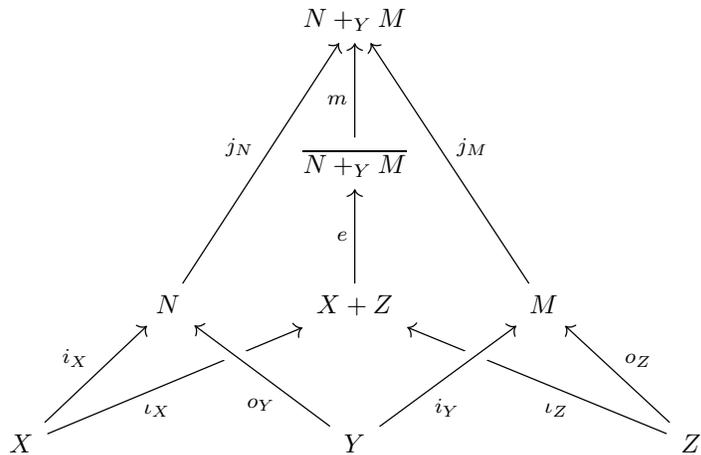
\[
  \begin{tikzcd}[row sep=7ex,column sep=7ex]
    && N+_YM \\
    && \overline{N+_YM} \arrow[u,pos=.4,"m"] \\
    & N \arrow[uur,"j_N"] & X+Z \arrow[from=dll,pos=.35,swap,"\iota_X"]
    \arrow[from=drr,pos=.35,"\iota_Z"]
    \arrow[u,"e"] & 
    M \arrow[uul,swap,"j_M"] \\
    X \arrow[ur,"i_X"] && Y
    \arrow[ul,crossing over,pos=.35,"o_Y"] \arrow[ur,crossing
    over,pos=.35,swap,"i_Y"] && Z,
    \arrow[ul,swap,"o_Z"]
  \end{tikzcd}
\]
where $m \circ e$ is the $(\mc E,\mc M)$-factorisation of $[j_N\circ i_X,j_M
\circ o_Z]\maps X+Z \to N+_YM$. It is not difficult to show, that this composite
is unique up to isomorphism. 

For nice categorical properties, like associativity under composition, it is
important that our factorisation system be costable.
\begin{definition}
  Given a category $\mc C$, we say that a subcategory $\mc M$ is \define{stable
  under pushout} if for every pushout square
  \[
    \xymatrixcolsep{3pc}
    \xymatrixrowsep{3pc}
    \xymatrix{
      \ar[r]^j & \\
      \ar[u] \ar[r]^m &  \ar[u]
    }
  \]
  such that $m \in \mathcal M$, we also have that $j \in \mathcal M$. We say
  that a factorisation system $(\mc E,\mc M)$ is \define{costable} if $\mc M$ is
  stable under pushout.
\end{definition}

\begin{proposition} \label{prop.corelcat}
  Let $\mc C$ be a category with finite colimits and a costable factorisation
  system $(\mc E,\mc M)$. Then there exists a category $\corel_{(\mc E,\mc
  M)}(\mc C)$ with the objects of $\mc C$ as objects, $(\mc E,\mc
  M)$-corelations as morphisms, and composition given as above. Moreover, the
  map taking a cospan to its $\mc E$-part defines a functor $\square\maps
  \cospan(\mc C) \to \corel(\mc C)$.
\end{proposition}

We will drop explicit reference to the factorisation system when context
allows, simply writing $\mathrm{Corel}(\mc C)$.

This is a standard result. For instance, a bicategorical version of the dual
theorem, for relations, can be found in \cite{JW00}. For more intuition
regarding corelations and their relationship to special commutative Frobenius
monoids, see \cite{CF17}.

\begin{examples} \label{ex.factsysts}
  Write $\mathcal I_{\mathcal C}$ for the wide subcategory of $\mathcal C$
  containing exactly the isomorphisms of $\mathcal C$. Two seemingly trivial,
  but important, examples of costable factorisation systems are $(\mathcal
  I_{\mathcal C}, \mathcal C)$ and $(\mathcal C, \mathcal I_{\mathcal C})$. The
  category $\corel_{(\mathcal I_{\mc C}, \mc C)}(\mc C)$ is equivalent to the
  terminal category, while $\corel_{(\mc C,\mc I_{\mc C})}(\mc C)$ is isomorphic
  to $\cospan(\mc C)$. 

  Another example of costable factorisation system is the epi-mono factorisation
  system $(\mathrm{Sur},\mathrm{Inj})$ in $\Set$, whence corelations $X \to Y$
  are equivalence relations on $X+Y$. 
  
  This generalises to any topos.  Indeed, Lack and Soboci\'nski showed that
  monomorphisms are stable under pushout in any adhesive category \cite{LS04}.
  Since any topos is both a regular category and an adhesive category
  \cite{LS06,Lac11}, the regular epimorphism-monomorphism factorisation system
  in any topos is costable.
  
  Another class of examples comes from coregular categories. A coregular
  category is by definition a category that has finite colimits and a costable
  epimorphism-regular monomorphism factorisation system. Examples of these
  include the category of topological spaces and continuous maps, as well as
  $\Set^\opp$, any cotopos, and so on.
\end{examples}

\section{Corelations form hypergraph categories} \label{sec.corels}

The focus of this paper is not just the construction of categories, but
\emph{hypergraph} categories.  In fact, all corelation categories come equipped
with this extra structure. In this section we explain the relevant data, and
tackle some of the monoidal considerations. We shall complete the proof that we
have a hypergraph category simultaneously with our consideration of functors in
the next section.

The hypergraph structure on $\corel(\mc C)$ is that which makes the canonical functor
\[
  \square\maps \cospan(\mc C) \longrightarrow \corel(\mc C)
\]
a hypergraph functor. Indeed, we define the coherence and Frobenius maps of
$\corel(\mc C)$ to be their image under this map. For the monoidal product we
again use the coproduct in $\mc C$; the monoidal product of two corelations is
their monoidal product as cospans.

\begin{theorem} \label{thm.cospantocorel}
  Let $\mathcal C$ be a category with finite colimits, and let $(\mathcal E,
  \mathcal M)$ be a costable factorisation system. Then there exists a hypergraph category
  $\mathrm{Corel}(\mathcal C)$ with 
  \begin{center}
    \begin{tabular}{| c | p{.65\textwidth} |}
      \hline
      \multicolumn{2}{|c|}{The hypergraph category $\big(\mathrm{Corel}_{(\mc
      E,\mc M)}(\mc C),+\big)$} \\
      \hline
      \textbf{objects} & the objects of $\mathcal C$ \\ 
      \textbf{morphisms} & isomorphism classes of $(\mc E,\mc M)$-corelations in $\mathcal C$\\ 
      \textbf{composition} & given by the $\mc E$-part of pushout \\
      \textbf{monoidal product} & the coproduct in $\mathcal C$ \\
      \textbf{coherence maps} & inherited from $\cospan(\mc C)$  \\
      \textbf{hypergraph maps} & inherited from $\cospan(\mc C)$ \\
      \hline
    \end{tabular}
  \end{center}  
  \smallskip
\end{theorem}

\subsection*{Proof strategy:} 
We will prove this theorem in two stages. The first stage, which will be the
rest of this section, is focussed on monoidal considerations. We prove two
lemmas, which respectively show that $\mc E$ and $\mc M$ are closed under $+$.
In particular, that $\mc E$ is closed (Lemma \ref{lem.monfact}) implies that the
proposed monoidal product on $\corel(\mc C)$ is independent of choice of
representative corelation, and hence well defined as a function. For the second
stage, it remains to check a number of axioms: functoriality of the monoidal
product, naturality of the coherence maps, the coherence axioms for symmetric
monoidal categories, the Frobenius laws. We do this in the next section.

Our strategy for the axiom checking of Stage 2 will be to show that the
surjective function from cospans to corelations defined by taking a cospan to
its jointly $\mc E$-part preserves both composition and the monoidal product.
This then implies that to evaluate an expression in the monoidal category of
corelations, we may simply evaluate it in the monoidal category of cospans, and
then take the $\mc E$-part. Thus if an equation is true for cospans, it is true
for corelations.

Instead of proving just this, however, we will prove a generalisation regarding
an analogous map between any two corelation categories. Such a map exists
whenever we have two corelation categories $\corel_{(\mc E,\mc M)}(\mc C)$ and
$\corel_{(\mc E',\mc M')}(\mc C')$ and a colimit
preserving functor $A\maps \mc C \to \mc C'$ such that the image of $\mc M$ lies
in $\mc M'$. As $(\mc C,\mc I_{\mc C})$-corelations are just cospans, this
reduces to the desired special case by taking the domain to be the category of
$(\mc C,\mc I_{\mc C})$-corelations, $\mc C'$ to be equal to $\mc C$, and $A$ to
be the identity functor. But the generality is not spurious: it has the
advantage of proving the existence of a class of hypergraph functors between
corelation categories in the same fell swoop. Although a touch convoluted, this strategy is worth the pause for thought. We
will use it once again for \emph{decorated} corelations, to great economy.
\smallskip

First though, back to Stage 1: monoidal considerations. As we are concerned with
building monoidal categories of corelations, it will be important that our
factorisation systems are so-named monoidal factorisation systems. These are
factorisation systems $(\mc E,\mc M)$ such that $(\mc E,\ot)$ is a monoidal
category. Luckily, when the monoidal product is the coproduct, \emph{all}
factorisation systems are monoidal factorisation systems.

\begin{lemma} \label{lem.monfact}
  Let $\mc C$ be a category with finite coproducts, and let $(\mc E, \mc M)$ be a
  factorisation system on $\mc C$. Then $(\mc E,+)$ is a symmetric monoidal
  category.
\end{lemma}
\begin{proof}
  The only thing to check is that $\mc E$ is closed under $+$. That is, given
  $f\maps A \to B$ and $g\maps C \to D$ in $\mc E$, we wish to show that
  $f+g\maps A+C \to B+D$, defined in $\mc C$, is also a morphism in $\mc E$. 

  Let $f+g$ have factorisation $A+C \stackrel{e}\longrightarrow \overline{B+D}
  \stackrel{m}\longrightarrow B+D$, where $e \in \mc E$ and $m \in \mc
  M$. We will prove that $m$ is an isomorphism. To construct an inverse, recall
  that by definition, as $f$ and $g$ lie in $\mc E$, there exist morphisms
  $x\maps B \to \overline{B+D}$ and $y\maps D \to \overline{B+D}$ such that
  \[ \label{eq.coreltensor}
    \begin{gathered}
    \xymatrixcolsep{2pc}
    \xymatrixrowsep{2pc}
    \xymatrix{
      A \ar[r]^f \ar[d] & B \ar@{=}[r] \ar@{.>}[d]^x & B
      \ar[d] \\
      A+C \ar[r]_{e}&\overline{B+D} \ar[r]_{m} & B+D
    }
  \end{gathered}
    \qquad \mbox{and} \qquad
    \begin{gathered}
    \xymatrixcolsep{2pc}
    \xymatrixrowsep{2pc}
    \xymatrix{
      C \ar[r]^g \ar[d] & D \ar@{=}[r] \ar@{.>}[d]^y & D
      \ar[d] \\
      A+C \ar[r]_{e}&\overline{B+D} \ar[r]_{m} & B+D
    }
  \end{gathered}
    \tag{1}
  \]
  The copairing $[x,y]$ is an inverse to $m$. 
  
  Indeed, taking the coproduct of the top rows of the two diagrams above and the
  copairings of the vertical maps gives the commutative diagram
  \[
    \xymatrix{
      A+C \ar[r]^{f+g} \ar@{=}[d] & B+D \ar@{=}[r] \ar[d]_{[x,y]} & B+D \ar@{=}[d] \\
      A+C \ar[r]^{e} & \overline{B+D} \ar[r]^{m} & B+D
    }
  \]
  Reading the right-hand square immediately gives $m \circ [x,y] =1$.
  
  Conversely, to see that $[x,y] \circ m = 1$, remember that by definition $f+g
  = m \circ e$. So the left-hand square above implies that
  \[
    \xymatrixcolsep{2pc}
    \xymatrixrowsep{2pc}
    \xymatrix{
      A+C \ar[r]^e \ar@{=}[d] & \overline{B+D} \ar[d]^{[x,y] \circ m} \\
      A+C \ar[r]_{e}&\overline{B+D} 
    }
  \]
  commutes. But by the universal property of factorisation systems, there is a
  unique map $\overline{B+D} \to \overline{B+D}$ such that this diagram
  commutes, and clearly the identity map also suffices. Thus $[x,y] \circ m =
  1$.
\end{proof}

The analogous fact for $\mc M$ is also important. It follows from stability under pushout.

\begin{lemma} \label{lem.mcoproductsmc}
  Let $\mathcal C$ be a category with finite colimits, and let $\mathcal M$ be a
  subcategory of $\mathcal C$ stable under pushouts and containing all
  isomorphisms. Then $(\mc M,+)$ is a symmetric monoidal category.
\end{lemma}
\begin{proof}
  It is enough to show that for all morphisms $m,m' \in \mc M$ we have $m+m'$ in
  $\mc M$. Since $\mc M$ contains all isomorphisms, the coherence maps are
  inherited from $\mc C$. The required axioms---the functoriality of the tensor
  product, the naturality of the coherence maps, and the coherence laws---are
  also inherited as they hold in $\mc C$.

  To see $m+m'$ is in $\mc M$, simply observe that we have the pushout square
  \[
  \xymatrixcolsep{3pc}
  \xymatrixrowsep{2pc}
    \xymatrix{
      A+C \ar[r]^{m+1} & B+C \\
      A \ar[r]^m \ar[u]^{\iota} & B \ar[u]_{\iota} \\
    }
  \]
  in $\mc C$. As $\mc M$ is stable under pushout, $m+1 \in \mc M$. Similarly,
  $1+m' \in \mc M$. Thus their composite $m+m'$ lies in $\mc M$, as required.
\end{proof}

\begin{remark}
An analogous argument shows that pushouts of maps $m+_Ym'$ also lie in $\mc M$.
Using this fact it is not difficult to show the associativity of composition of
corelations---the key point is that factorisation commutes with pushouts. 
\end{remark}

\section{Functors between corelation categories} \label{sec.corelfunctors}

To construct a functor between cospan categories one may start with a
colimit-preserving functor between the underlying categories. Corelations are
cospans where we forget the $\mc M$-part of each cospan. Hence for functors
between corelation categories, we require not just a colimit-preserving functor
but, loosely speaking, also that we don't forget more in the domain category
than in the codomain category.

We devote the next few pages to proving the following proposition. Along the way
we prove, as promised, that corelation categories are well-defined hypergraph
categories.

\begin{proposition} \label{prop.corelfunctors}
  Let $\mathcal C$, $\mathcal C'$ have finite colimits and respective costable
  factorisation systems $(\mathcal E, \mathcal M)$, $(\mathcal E', \mathcal
  M')$. Further let $A\maps \mathcal C \to \mathcal C'$ be a functor that
  preserves finite colimits and such that the image of $\mathcal M$ lies in
  $\mathcal M'$.

  Then we may define a hypergraph functor $\square\maps \corel(\mathcal C) \to
  \corel(\mathcal C')$ sending each object $X$ in $\corel(\mathcal C)$ to $AX$
  in $\corel(\mc C')$ and each corelation 
  \[
    X \stackrel{i_X}{\longrightarrow} N \stackrel{o_Y}{\longleftarrow} Y 
  \]
  to the $\mc E'$-part
  \[
    AX \xrightarrow{e'\circ\iota_{AX}} \overline{AN}
    \xleftarrow{e'\circ\iota_{AY}} AY.
  \]
  of the image cospan. The coherence maps are the $\mc E'$-part
  $\overline{\kappa_{X,Y}}$ of the isomorphisms $\kappa_{X,Y}\maps AX+AY \to
  A(X+Y)$ given as $A$ preserves colimits.
\end{proposition}

As discussed, we still have to prove that $\corel(\mc C)$ is a hypergraph
category. We address this first with two lemmas regarding these proposed
functors.

\begin{lemma} \label{lem.corelfuncomposition}
  The above function $\square\maps \corel(\mc C) \to \corel(\mc C')$ preserves
  composition.
\end{lemma}
\begin{proof}
  Let $f = (X \longrightarrow N \longleftarrow Y)$ and $g= (Y \longrightarrow M
  \longleftarrow Z)$ be corelations in $\mathcal C$. By definition, the
  corelations $\square(g) \circ \square(f)$ and $\square(g \circ f)$ are given
  by the first arrows in the top and bottom row respectively of the diagram:
  \[ \label{diag.eparts}
    \begin{aligned}
      \xymatrixcolsep{5.5pc}
      \xymatrixrowsep{2pc}
      \xymatrix{
	\scriptstyle AX+AZ \ar[r]^{\mc E'} \ar@{=}[d] & \scriptstyle \overline{\overline{AN}+_{AY}\overline{AM}}
	\ar[r]^{\mc M'} \ar@{<.>}[d]^{n} & \scriptstyle \overline{AN}+_{AY}\overline{AM}
	\ar[r]^{m'_{AN}+_{AY}m'_{AM}} & \scriptstyle
	AN+_{AY}AM \\
	\scriptstyle AX+AZ \ar[r]^{\mc E'} & \scriptstyle \overline{A(\overline{N+_YM})} \ar[r]^{\mc M'} & \scriptstyle
	A(\overline{N+_YM}) \ar[r]^{Am_{N+_YM}} & \scriptstyle A(N+_YM) \ar@{<->}[u]_{\sim}
      }
    \end{aligned}
    \tag{$\ast$}
  \]
  The morphisms labelled $\mc E'$ lie in $\mc E'$, and similarly for $\mc M'$;
  these are given by the factorisation system on $\mc C'$.  The maps
  $Am_{N+_YM}$ and $m'_{AN}+_{AY}m'_{AM}$ lie in $\mc M'$ too: $Am_{N+_YM}$ as
  it is in the image of $\mc M$, and $m'_{AN}+_{AY}m'_{AM}$ as $\mc M'$ is
  stable under pushout. 

  Moreover, the diagram commutes as both maps $AX+AZ \to AN+_{AY}AM$ compose to
  that given by the pushout of the images of $f$ and $g$ over $AY$.  Thus the
  diagram represents two $(\mc E', \mc M')$ factorisations of the same morphism,
  and there exists an isomorphism $n$ between the corelations $\square(g) \circ
  \square(f)$ and $\square(g\circ f)$. This proves that $\square$ preserves
  composition.
\end{proof}

\begin{remark}
  While we have already assumed that $\corel(\mc C)$ is a category, this first
  lemma allows us to verify the associativity and unit laws for $\corel(\mc C)$.
  Consider the case of Proposition \ref{prop.corelfunctors} with $\mc C = \mc
  C'$, $(\mc E,\mc M) = (\mc C, \mc I_{\mc C})$, and $A = 1_{\mc C}$. Then the
  domain of $\square$ is $\cospan(\mc C)$ by definition. (Indeed, $\square$ is
  the functor of Proposition \ref{prop.corelcat} mapping a cospan to its $\mc
  E$-part.) In this case, the function $\square\maps \cospan(\mc C) \to
  \corel(\mc C)$ is bijective-on-objects and surjective-on-morphisms. Thus to
  compute the composite of any two corelations, we may consider them as cospans,
  compute their composite \emph{as cospans}, and then take the $\mc E$-part of
  the result. Since composition of cospans is associative and unital, so is
  composition of corelations, with the identity corelation just the image of the
  identity cospan.
\end{remark}

This first lemma is useful in proving an second important lemma: the
naturality of $\overline{\kappa}$.

\begin{lemma} \label{lem.corelfunmonoidal}
  The maps $\overline{\kappa_{X,Y}}$, as defined in Proposition
  \ref{prop.corelfunctors}, are natural.
\end{lemma}
\begin{proof}
  Let $f = (X \longrightarrow N \longleftarrow Y)$, $g= (Z \longrightarrow M
  \longleftarrow W)$ be corelations in $\mc C$. We wish to show that
  \[
    \xymatrixcolsep{4pc}
    \xymatrixrowsep{3pc}
    \xymatrix{
      AX+AY \ar[r]^{\square(f)+\square(g)}
      \ar[d]_{\overline{\kappa_{X,Y}}} & 
      AZ+AW \ar[d]^{\overline{\kappa_{Z,W}}} \\
      A(X+Y) \ar[r]^{\square(f+g)} & A(Z+W)
    }
  \]
  commutes in $\corel(\mc C')$. 

  Consider the following commutative diagram in $\mc C'$, with the outside
  square equivalent to the naturality square for the coherence maps of the
  monoidal functor \linebreak $\cospan(\mc C) \to \cospan(\mc C')$:
  \[ \label{diag.natural}
    \begin{aligned}
      \xymatrixcolsep{4pc}
      \xymatrixrowsep{2.5pc}
      \xymatrix{
	(AX+AY)+(AZ+AW) \ar[r]^(.65){\mc E'+\mc E'}
	\ar[d]_{\kappa_{X,Y}+\kappa_{Z,W}} & 
	\overline{AN}+\overline{AM} \ar[r]^{\mc M'+\mc M'} \ar@{.>}[d]^{p} & 
	AN+AM \ar[d]^{\kappa_{N,M}}\\
	A(X+Y)+A(Z+W) \ar[r]^{\mc E'} & \overline{A(N+M)} \ar[r]^{\mc M'} & A(N+M)
      }
    \end{aligned}
    \tag{$\#$}
  \]
  We have factored the top edge as the coproduct of the respective
  factorisations of $f$ and $g$, and the bottom edge simply as the factorisation
  of the coproduct $f+g$. 

  Note that by Lemma \ref{lem.monfact} the coproduct of two maps in $\mc E'$ is
  again in $\mc E'$, while Lemma \ref{lem.mcoproductsmc} implies the same for
  $\mc M'$. Thus the top edge is an $(\mc E',\mc M')$-factorisation, and the
  uniqueness of factorisations gives the isomorphism $n$. 
  Given that the map reducing cospans to corelations is functorial, the
  commutative square
  \[
    \xymatrixcolsep{2.5pc}
    \xymatrixrowsep{2.5pc}
    \xymatrix{
      (AX+AY)+A(Z+W) \ar[r]^{1+\kappa_{Z,W}^{-1}} \ar@{=}[d] & (AX+AY)+(AZ+AW)
      \ar[r]^(.65){\mc E'+\mc E'} & 
      \overline{AN}+\overline{AM} \ar[d]^{n} \\
      (AX+AY)+A(Z+W) \ar[r]^{\kappa_{X,Y}+1} & A(X+Y)+A(Z+W) \ar[r]^(.6){\mc E'} & 
      \overline{A(N+M)}
    }
  \]
  then implies the naturality of the maps $\overline{\kappa}$.
\end{proof}

These lemmas now imply that $\corel(\mc C)$ is a well-defined hypergraph
category.
\begin{proof}[of Theorem \ref{thm.cospantocorel}.]
  To complete the proof, consider again the case of
  Proposition~\ref{prop.corelfunctors} with $\mc C = \mc C'$, $(\mc E,\mc M) =
  (\mc C, \mc I_{\mc C})$, and $A = 1_{\mc C}$. Note that by definition this
  function maps the coherence and hypergraph maps of $\cospan(\mc C)$ onto the
  corresponding maps of $\corel(\mc C)$. Then since $\cospan(\mc C)$ is a
  hypergraph category, and since $\square$ preserves composition and respects
  the monoidal and hypergraph structure, $\corel(\mc C)$ is also a hypergraph
  category. 
  
  For instance, suppose we want to check the functoriality of the monoidal
  product $+$. We then wish to show $(g \circ f) + (k \circ h) = (g + k) \circ
  (f + h)$ for corelations of the appropriate types.  But $\square$ preserves
  composition, and the naturality of $\kappa$, here the identity map, implies
  that for any two cospans the $\mc E$-part of their coproduct is equal to the
  coproduct of their $\mc E$-parts. Thus we may compute these two expressions by
  viewing $f$, $g$, $h$, and $k$ as cospans, evaluating them in the category of
  cospans, and then taking their $\mc E$-parts. Since the equality holds in the
  category of cospans, it holds in the category of corelations.
\end{proof}

\begin{corollary}
  The functor 
  \[
    \square\maps \mathrm{Cospan}(\mathcal C) \longrightarrow
    \mathrm{Corel}(\mathcal C),
  \]
  that takes each object of $\cospan(\mathcal C)$ to itself as an object of
  $\corel(\mathcal C)$ and each cospan to its $\mathcal E$-part is a strict
  hypergraph functor. 
\end{corollary}

Finally, we complete the proof that $\square\maps \corel(\mc C) \to \corel(\mc
C')$ is in general a hypergraph functor.

\begin{proof}[of Proposition \ref{prop.corelfunctors}.] 
  We show $\square$ is a functor, a symmetric monoidal functor, and then finally
  a hypergraph functor.

  \paragraph{Functoriality.} First, recall that $\square$ preserves composition
  (Lemma \ref{lem.corelfuncomposition}). Thus to prove $\square$ is a functor it
  remains to show identities are mapped to identities. The general idea for this
  and for similar axioms is to recall that the special maps are given by reduced
  versions of particular colimits, and that $(\mc E',\mc M')$ reduces maps more
  than $(\mc E,\mc M)$. 

  In this case, recall the identity corelation is given by the $\mc E$-part $X+X
  \to \overline{X}$ of $[1,1]\maps X+X \to X$. Thus the image of the identity on
  $X$ and the identity on $AX$ are given by the top and bottom rows of the
  commuting square
  \[
    \xymatrixcolsep{4pc}
    \xymatrixrowsep{2pc}
    \xymatrix{
      A(X+X) \ar[d]^{\kappa^{-1}}_{\sim} \ar[r]^{\mc E'} &
      \overline{A\overline{X}} \ar[r]^{\mc M'} \ar@{.>}[d]^{n} &
      A\overline{X} \ar[r]^{A\mc M} & AX \ar@{=}[d]\\
      AX+AX \ar[r]^{\mc E'} & \overline{AX} \ar[rr]^{\mc M'} && AX
    }
  \]
  The outside square commutes as we know $A$ maps the identity cospan of $\mc C$
  to the identity cospan of $\mc C'$. The top row is the image under $A$ of the
  identity cospan in $\mc C$, factored first in $\mc C$, and then in $\mc C'$.
  The bottom row is just the factored identity cospan on $AX$ in $\mc C'$. As
  $A$ maps $\mc M$ into $\mc M'$, the map marked $A\mc M$ lies in $\mc M'$. Thus
  both rows are $(\mc E',\mc M')$-factorisations, and so we have the isomorphism
  $n$. Thus $\square$ preserves identities.

  \paragraph{Strong monoidality.} We proved in Lemma \ref{lem.corelfunmonoidal}
  that our proposed coherence maps are natural. The rest of the properties
  follow from the composition preserving map $\cospan(\mc C') \to \corel(\mc
  C')$.  Since the $\kappa$ obey all the required axioms as cospans, they obey
  them as corelations too.

  \paragraph{Hypergraph structure.} The proof of preservation of the hypergraph
  structure follows the same pattern as the identity maps. 
\end{proof}

\begin{remark} \label{rem.corelposet}
  On any category $\mc C$ with finite colimits, reverse inclusions of the right
  factor $\mc M$ defines a partial order on the set of costable factorisation
  systems. That is, we write $(\mc E,\mc M) \ge (\mc E',\mc M')$ whenever $\mc M
  \subseteq \mc M'$.  The trivial factorisation systems $(\mc C,\mc I_{\mc C})$
  and $(\mc I_{\mc C},\mc C)$ are the top and bottom elements of this poset
  respectively.

  Corelation categories realise this poset as a subcategory of the category of
  hypergraph categories. One way to understand this is that corelations are
  cospans with the $\mc M$-part `forgotten'. Using the morphism-isomorphism
  factorisation system nothing is forgotten, so these corelations are just
  cospans. Using the isomorphism-morphism factorisation system everything is
  forgotten, so there is a unique corelation between any two objects.

  We can construct a hypergraph functor between two corelation categories
  precisely when the codomain forgets more than the domain: i.e. if the codomain
  is less than the domain in the poset. In particular, this implies there is
  always a hypergraph functor from the cospan category $\corel_{(\mc C,\mc
  I_{\mc C})}(\mc C)= \cospan(\mc C)$ to any other corelation category
  $\corel_{(\mc E,\mc M)}(\mc C)$, and from $\corel_{(\mc E,\mc M)}(\mc C)$
  any corelation category to the indiscrete category $\corel_{(\mc I_{\mc
  C},\mc C)}(\mc C)$ on the objects of $\mc C$.
\end{remark}

\section{Decorated corelations} \label{sec.dcorc}

In this section we define the category of decorated corelations. 

Recall that decorating cospans requires more than just choosing a set of
decorations for each apex: for composition, we need to describe how these
decorations transfer along the copairing of pushout maps $[j_N,j_M]\maps N+M \to
N+_YM$. Thus to construct a decorated cospan category we need not merely a
function from the objects of $\mc C$ to $\Set$, but a lax symmetric monoidal
functor $(\mc C,+) \to (\Set,\times)$. 

Similarly, decorating $(\mc E,\mc M)$-corelations requires still more
information: we now further need to know how to transfer decorations backwards along the
morphisms $N+_YM \xleftarrow{m} \overline{N+_YM}$. We thus introduce the
symmetric monoidal category $\mc C;\mc M^\opp$ with morphisms isomorphism
classes of cospans of the form $\xrightarrow{f}\xleftarrow{m}$, where $f \in \mc
C$ and $m \in \mc M$.  For constructing categories of decorated $(\mc E,\mc
M)$-corelations, we then require a lax symmetric monoidal functor $F$ from $\mc
C;\mc M^\opp$ to $\Set$. 

To prove that we have indeed defined a hypergraph category of decorated
corelations, we will proceed as we did for corelations, using
structure-preserving functions from a category already known to be hypergraph.
This will hence again be completed in our discussion of functors in the next
section.
\smallskip

\subsection*{Adjoining right adjoints} \label{ssec.rightads}\hfill

\noindent
Suppose we have a cospan $X+Y \to N$ with a decoration on $N$. Reducing this to
a corelation requires us to factor this to $X+Y \stackrel{e}\to \overline{N}
\stackrel{m}\to N$. To define a category of decorated corelations, then, we must
specify how to take decoration on $N$ and `pull it back' along $m$ to a decoration on
$\overline{N}$.

For decorated cospans, it is enough to have a functor $F$ from a category $\mc
C$ with finite colimits; the image $Ff$ of morphisms $f$ in $\mc C$ describes
how to move decorations forward along $f$. We now want to expand $\mc C$ to
include a morphism $m^\opp$ for each $m$ in $\mc M$, so that the image $Fm^\opp$
describes how to move decorations backwards along $m$.  This is allowed by the
stability of $\mc M$ under pushouts.

\begin{proposition}
  Let $\mathcal C$ be a category with finite colimits, and let $\mathcal M$ be a
  subcategory of $\mathcal C$ stable under pushouts. Then we define the category
  $\mathcal C; \mathcal M^\opp$ as follows  
  \smallskip 

  \begin{center}
    \begin{tabular}{| c | p{.65\textwidth} |}
      \hline
      \multicolumn{2}{|c|}{The symmetric monoidal category $(\mc C;\mc M^\opp,+)$} \\
      \hline
      \textbf{objects} & the objects of $\mathcal C$ \\ 
      \textbf{morphisms} & isomorphism classes of cospans of the form
      $\stackrel{c}\rightarrow \stackrel{m}\leftarrow$, where $c$ lies in
      $\mathcal C$ and $m$ in $\mathcal M$\\ 
      \textbf{composition} & given by pushout \\
      \textbf{monoidal product} & the coproduct in $\mathcal C$ \\
      \textbf{coherence maps} & the coherence maps in $\mc C$ \\
      \hline
    \end{tabular}
  \end{center}
  \smallskip 
\end{proposition}

\begin{proof}
  Our data is well defined: composition because $\mc M$ is stable under
  pushouts, and monoidal composition by Lemma \ref{lem.mcoproductsmc}. The
  coherence laws follow as this is a symmetric monoidal subcategory of
  $\cospan(\mc C)$. 
\end{proof}

\begin{remark}
As we state in the proof, the category $\mc C;\mc M^\opp$ is a subcategory of
$\cospan(\mc C)$. We can in fact view it as a sub-bicategory of the bicategory
of cospans in $\mc C$, where the 2-morphisms given by maps of cospans. In this
bicategory every morphism of $\mc M$ has a right adjoint.
\end{remark}

\begin{examples} 
  Note that $\mathcal C; \mathcal C^\opp$ is by definition equal to
  $\cospan(\mathcal C)$, $\mathcal C;\mathcal I_{\mathcal C}^\opp$ is isomorphic
  to $\mathcal C$, and $\Set;\mathrm{Inj}^\opp$ is the category with sets as
  objects and partial functions as morphisms.
\end{examples}

The following lemma details how to construct functors between this type of
category.

\begin{lemma} \label{lem.madjointsfunctor}
  Let $\mathcal C$, $\mathcal C'$ be categories with finite colimits, and let
  $\mathcal M$, $\mathcal M'$ be subcategories of $\mc C$, $\mc C'$ respectively
  each stable under pushouts. Let $A\maps \mathcal C \to \mc C'$ be functor that
  preserves colimits and such that the image of $\mc M$ lies in $\mc M'$. Then
  $A$ extends to a symmetric strong monoidal functor
  \[
    A\maps \mc C;\mc M^\opp \longrightarrow \mc C'; \mc M'^\opp.
  \]
  mapping $X$ to $AX$ and $\stackrel{c}\rightarrow \stackrel{m}\leftarrow$ to
  $\xrightarrow{Ac} \xleftarrow{Am}$.
\end{lemma}
\begin{proof}
  Note $A(\mc M) \subseteq \mc M'$, so $\xrightarrow{Ac} \xleftarrow{Am}$ is
  indeed a morphism in $\mc C';\mc M'^\opp$. This is then a restriction and
  corestriction of the usual functor $\cospan(\mc C) \to \cospan(\mc C')$ to the
  above domain and codomain.
\end{proof}

Note that a similar construction giving subcategories of cospan categories could
be defined more generally using any two isomorphism-containing wide
subcategories stable under pushout.  The above, however, suffices for decorated
corelations.
\smallskip

\subsection*{Decorated corelations} \label{ssec.deccorel}\hfill

\noindent
As we have said, decorated corelations are constructed from a lax symmetric
monoidal functor from $\mc C;\mc M^\opp$ to $\Set$. We now define
decorated corelations and give a composition rule for them, showing that this
composition rule is well defined up to isomorphism.

\begin{definition}
  Let $\mathcal C$ be a category with finite colimits, $(\mathcal E, \mathcal
  M)$ be a costable factorisation system, and 
  \[
    F: (\mathcal C;\mathcal M^\opp,+) \longrightarrow (\Set, \times)
  \]
  be a lax symmetric monoidal functor.  We define an $F$-\define{decorated
  corelation} to a pair
  \[
    \left(
    \begin{aligned}
      \xymatrix{
	& N \\  
	X \ar[ur]^{i} && Y \ar[ul]_{o}
      }
    \end{aligned}
    ,
    \qquad
    \begin{aligned}
      \xymatrix{
	FN \\
	1 \ar[u]_{s}
      }
    \end{aligned}
    \right)
  \]
  where the cospan is jointly $\mathcal E$-like. A morphism of decorated
  corelations is a morphism of decorated cospans between two decorated
  corelations. 
\end{definition}

Suppose we have decorated corelations
\[
  \left(
  \begin{aligned}
    \xymatrix{
      & N \\  
      X \ar[ur]^{i_X} && Y \ar[ul]_{o_Y}
    }
  \end{aligned}
  ,
  \qquad
  \begin{aligned}
    \xymatrix{
      FN \\
      1 \ar[u]_{s}
    }
  \end{aligned}
  \right)
  \qquad
  \mbox{and}
  \qquad
  \left(
  \begin{aligned}
    \xymatrix{
      & M \\  
      Y \ar[ur]^{i_Y} && Z \ar[ul]_{o_Z}
    }
  \end{aligned}
  ,
  \qquad
  \begin{aligned}
    \xymatrix{
      FM \\
      1 \ar[u]_{t}
    }
  \end{aligned}
  \right).
\]
Then, recalling the notation introduced in \textsection\ref{sec.background},
their composite is given by the composite corelation
\[
  \xymatrix{
    & \overline{N+_YM} \\  
    X \ar[ur]^{e\circ \iota_X} && Z \ar[ul]_{e \circ \iota_Z}
  }
\]
paired with the decoration
\[
  1 \xrightarrow{\varphi_{N,M}\circ\langle s,t\rangle} F(N+M)
  \xrightarrow{F[j_N,j_M]} F(N+_YM) \xrightarrow{F(m^\opp)} F(\overline{N+_YM}).
\]
As composition of corelations and decorated cospans are both well defined up to
isomorphism, it is straightforward to show that this too is well defined up to isomorphism. 

\begin{proposition} \label{prop.deccorelcomp}
  The above is a well-defined composition rule on isomorphism classes of
  decorated corelations.
\end{proposition}
\begin{proof}
Let 
\[
  \big(X \stackrel{i_X}\longrightarrow N \stackrel{o_Y}\longleftarrow Y,\enspace 1
\stackrel{s}\longrightarrow FN\big) \stackrel\sim\longrightarrow \big(X \stackrel{i'_X}\longrightarrow N'
\stackrel{o'_Y}\longleftarrow Y,\enspace 1 \stackrel{s'}\longrightarrow FN'\big)
\]
and
\[
  \big(Y \stackrel{i_Y}\longrightarrow M \stackrel{o_Z}\longleftarrow Z,\enspace 1
\stackrel{t}\longrightarrow FM\big) \stackrel\sim\longrightarrow \big(Y \stackrel{i'_Y}\longrightarrow M'
\stackrel{o'_Z}\longleftarrow Z,\enspace 1 \stackrel{t'}\longrightarrow FM'\big)
\]
be isomorphisms of decorated corelations. We wish to show that the composite of
the decorated corelations on the left is isomorphic to the composite of the
decorated corelations on the right. 

By definition, the composites of the underlying corelations are isomorphic, via
an isomorphism $s$ which exists by the factorisation system. We need to show
this $s$ is an isomorphism of decorations.  This is a matter of showing the
diagram 
\[
    \xymatrixcolsep{3pc}
    \xymatrixrowsep{1pc}
 \xymatrix{
   && F(N+_YM) \ar[r]^{Fm^\opp} \ar[dd]_\sim^{Fp} & F(\overline{N+_YM})
   \ar[dd]_\sim^{Fs} \\
   1 \ar[urr]^(.4){F[j_N,j_M]\circ\varphi_{N,M}\circ \langle s,t \rangle\qquad}
   \ar[drr]_(.4){F[j_{N'},j_{M'}]\circ\varphi_{N',M'}\circ \langle s',t'
 \rangle\qquad\quad} \\
   && F(N'+_YM') \ar[r]_{Fm'^\opp} & F(\overline{N'+_YM'})\\
 }
\]
The triangle commutes as composition of decorated cospans is well defined, while
the square commutes as composition of corelations is well defined.
\end{proof}

\begin{remark}
  We could give a more general definition of decorated corelation for lax
  braided monoidal functors 
  \[
    (\mathcal C;\mathcal M^\opp,+) \longrightarrow (\mc D, \otimes).
  \]
  A similar argument to that in Appendix \ref{ssec.setdecorations} shows,
  however, that we gain no extra generality. On the other hand, keeping track of
  this possibly varying category $\mc D$ in the following distracts from the
  main insights. We thus merely remark that it is possible to make the more
  general definition, and leave it at that.
\end{remark}
\smallskip

\subsection*{Categories of decorated corelations} \label{ssec.deccorelcats}\hfill

\noindent
We now define the hypergraph category $F\mathrm{Corel}$ of decorated
corelations. Having defined decorated corelations and their composition in the
previous subsection, the key question to address is the provenance of the
monoidal and hypergraph structure. 

Recall, from \textsection\ref{sec.corels}, that to define the monoidal and
hypergraph structure on categories of corelations, we used functors
$\mathrm{Cospan}(\mc C) \to \mathrm{Corel}(\mc C)$, leveraging the monoidal and
hypergraph structure on cospan categories. In analogy, here we leverage the same
fact for decorated cospans, this time using a structure preserving map 
\[
  \square\maps F\mathrm{Cospan} \longrightarrow F\mathrm{Corel}.
\]
Here $F\mathrm{Cospan}$ denotes the decorated cospan category constructed from
the restriction of the functor $F\maps \mc C;\mc M^\opp \to \Set$ to the domain
$\mc C$. 

The monoidal product of two decorated corelations is their monoidal product as
decorated cospans. To define the coherence maps for this monoidal product, as
well as the coherence maps, we introduce the notion of a restricted decoration.

Given a cospan $X \to N \leftarrow Y$, write $m\maps \overline{N} \to N$ for the
$\mc M$ factor of the copairing $X+Y \to N$. The map $\square$ takes a
decorated cospan 
\[
  (X \stackrel{i}\longrightarrow N \stackrel{o}\longleftarrow Y, \enspace 1
    \stackrel{s}\longrightarrow FN)
\]
to the decorated corelation 
\[
  (X \stackrel{\overline{i}}\longrightarrow \overline N
  \stackrel{\overline{o}}{\longleftarrow} Y, \enspace 1 \xrightarrow{Fm^\opp \circ
  s} F\overline{N}),
\]
where the corelation is given by the jointly $\mc E$-part of the cospan, and the
decoration is given by composing $s$ with the $F$-image $Fm^\opp\maps FN \to
F\overline{N}$ of the map $N \stackrel{1_N}\to N \stackrel{m}\leftarrow
\overline{N}$ in $\mc C;\mc M^\opp$. This is well-defined up to isomorphism of
decorated corelations. We call $Fm^\opp \circ s$ the \define{restricted
decoration} of the decoration on the cospan $(X \to N \leftarrow Y, \enspace 1
\stackrel{s}\to FN)$.

We then make the following definition.

\begin{theorem} \label{thm.fcorel}
  Let $\mathcal C$ be a category with finite colimits and a costable
  factorisation system $(\mathcal E, \mathcal M)$, and let 
  \[
    F: (\mathcal C;\mathcal M^\opp,+) \longrightarrow (\Set, \times)
  \]
  be a lax symmetric monoidal functor.  Then we may define 
  \smallskip 

  \begin{center}
    \begin{tabular}{| c | p{.65\textwidth} |}
      \hline
      \multicolumn{2}{|c|}{The hypergraph category $(F\mathrm{Corel},+)$} \\
      \hline
      \textbf{objects} & the objects of $\mathcal C$ \\ 
      \textbf{morphisms} & isomorphism classes of $F$-decorated corelations in
      $\mathcal C$\\ 
      \textbf{composition} & given by $\mc E$-part of pushout with restricted
      decoration  \\
      \textbf{monoidal product} & the coproduct in $\mathcal C$  \\
      \textbf{coherence maps} & maps from $\cospan(\mc C)$ with restricted empty
      decoration \\
      \textbf{hypergraph maps} & maps from $\cospan(\mc C)$ with restricted empty
      decoration \\
      \hline
    \end{tabular}
  \end{center}
  \smallskip 
\end{theorem}

Similar to Theorem~\ref{thm.cospantocorel} defining the hypergraph category
$\corel(\mc C)$, we have now specified well-defined data and just need to check a
collection of coherence axioms. As before, we prove this in the next section,
alongside a theorem regarding functors between decorated corelation categories.

\begin{remark}
  Decorated corelations generalise both decorated cospans, and corelations.
  Decorated cospans are simply decorated corelations with respect to the trivial
  factorisation system $(\mc C, \mc I_{\mc C})$.  `Undecorated' corelations are
  corelations decorated by the constant symmetric monoidal functor
  $\{\ast\}\maps \mc C;\mc M^\opp \to \Set$ on some terminal object $\{\ast\}$
  of $\Set$.
\end{remark}

\begin{remark}
  Note that decorated corelations are strictly more general than decorated
  cospans. For example, the category of epi-mono corelations in $\Set$ is not a
  decorated cospan category. 
  
  To see this, we count so-named scalars: morphisms from the monoidal unit
  $\varnothing$ to itself. In a decorated cospan category, the set of morphisms
  from $X$ to $Y$ always comprises all decorated cospans $(X \to N \leftarrow
  Y,\enspace 1 \to FN)$. Now for any object $N$ in the underlying category $\mc
  C$, there is a unique morphism $\varnothing \to N$. This means that the
  morphisms $\varnothing \to \varnothing$ are indexed by (isomorphism classes
  of) elements of $FN$, ranging over $N$.

  Suppose we have a decorated cospan category with a unique morphism
  $\varnothing \to \varnothing$. By the previous paragraph, and replacing $\mc
  C$ with an equivalent skeletal category, this implies there is only one object
  $N$ such that $FN$ is nonempty. But $FN$ must always contain at least one
  element, the empty decoration $1 \xrightarrow{\varphi_I} F\varnothing
  \xrightarrow{F!} FN$.  This implies there is only one object $N$ in $\mc C$:
  the object $\varnothing$.  Thus $\mc C$ must be the one object discrete
  category, and $F\maps \mc C \to \Set$ is the functor that sends the object of
  $\mc C$ to the one element set $1$.

  Hence any decorated cospan category with a unique morphism $\varnothing \to
  \varnothing$ is the one object discrete category. But the category of epi-mono
  corelations in $\Set$ is a nontrivial category with a unique morphism
  $\varnothing \to \varnothing$.  Thus it cannot be constructed as a decorated
  cospan category.

  On the other hand, as far as hypergraph categories are concerned, we need not
  get more general than decorated corelations: every hypergraph category is
  equivalent to a decorated corelation category \cite{Fon16}.
\end{remark}


\section{Functors between decorated corelation categories} \label{sec.dcorf}

In this section we show how to construct hypergraph functors between decorated
corelation categories. The construction of these functors holds no surprises:
their requirements combine the requirements of corelations and decorated
cospans. In the process of proving that our construction gives well-defined
hypergraph functors, we also complete the necessary prerequisite proof that
decorated corelation categories are well-defined hypergraph categories.

Recall that Lemma \ref{lem.madjointsfunctor} says that, when the image of $\mc
M$ lies in $\mc M'$, we can extend a colimit-preserving functor $\mc C \to \mc
C'$ to a symmetric monoidal functor $\mc C;\mc M^\opp \to \mc C';\mc M'^\opp$ .

\begin{proposition}\label{prop.deccorelfunctors}
  Let $\mathcal C$, $\mathcal C'$ have finite colimits and respective costable
  factorisation systems $(\mathcal E, \mathcal M)$, $(\mathcal E', \mathcal
  M')$, and suppose that we have lax symmetric monoidal functors
\[
  F: (\mathcal C;\mathcal M^\opp,+) \longrightarrow (\Set, \times)
\]
and
\[
  G: (\mathcal C';\mathcal M'^\opp,+) \longrightarrow (\Set, \times).
\]

Further let $A\maps \mathcal C \to \mathcal C'$ be a functor that preserves
finite colimits and such that the image of $\mathcal M$ lies in $\mathcal M'$.
This functor $A$ extends to a symmetric monoidal functor $\mc C;\mc M^\opp \to
\mc C';\mc M'^\opp$.

Suppose we have a monoidal natural transformation $\theta$:
\[
  \xymatrixrowsep{2ex}
  \xymatrix{
    \mc C; \mc M^\opp \ar[dd]_{A} \ar[drr]^F  \\
    &\twocell \omit{_\:\theta}& \Set \\
    \mc C'; \mc {M'}^\opp \ar[urr]_{G} 
  }
\]

Then we may define a hypergraph functor $T\maps F\mathrm{Corel} \to
G\mathrm{Corel}$ sending each object $X \in F\mathrm{Corel}$ to $AX \in
G\mathrm{Corel}$ and each decorated corelation 
  \[
    \left(
    \begin{aligned}
      \xymatrix{
	& N \\  
	X \ar[ur]^{i} && Y \ar[ul]_{o}
      }
    \end{aligned}
    ,
    \qquad
    \begin{aligned}
      \xymatrix{
	FN \\
	1 \ar[u]_{s}
      }
    \end{aligned}
    \right)
  \]  
to
  \[
    \left(
    \begin{aligned}
      \xymatrix{
	& \overline{AN} \\  
	AX \ar[ur]^{e'\circ\iota_{AX}} && AY \ar[ul]_{e'\circ\iota_{AY}}
      }
    \end{aligned}
    ,
    \qquad
    \begin{aligned}
      \xymatrixrowsep{1.5ex}
      \xymatrix{
	G\overline{AN} \\
	GAN \ar[u]_{Gm_{AN}^\opp}\\
        FN \ar[u]_{\theta_N}\\
	1 \ar[u]_{s}
      }
    \end{aligned}
    \right)
  \]  
  The coherence maps $\overline{\kappa_{X,Y}}$ are given by the coherence maps of
  $A$ with the restricted empty decoration.
\end{proposition}
\begin{proof}[of Theorem \ref{thm.fcorel} and Proposition
  \ref{prop.deccorelfunctors}.]
  In the proof of Theorem \ref{thm.cospantocorel} and Proposition
  \ref{prop.corelfunctors} we proved that the map 
  \[
    \square\maps \corel(\mc C) \longrightarrow \corel(\mc C')
  \]
  preserved composition and had natural coherence maps. Specialising to the case
  when $\corel(\mc C)=\cospan(\mc C')$, we saw that this bijective-on-objects,
  surjective-on-morphisms, composition and monoidal product preserving map
  proved $\corel(\mc C')$ is a hypergraph category, and it immediately followed
  that $\square$ is a hypergraph functor.

  The analogous argument holds here: we simply need to prove
  \[
    \square\maps F\mathrm{Corel} \longrightarrow G\mathrm{Corel}
  \]
  preserves composition and has natural coherence maps. Theorem \ref{thm.fcorel}
  then follows from examining the map $F\mathrm{Cospan} \to F\mathrm{Corel}$
  obtained by choosing $\mc C = \mc C'$, $(\mc E,\mc M) = (\mc C', \mc I_{\mc
  C'})$, $F$ the restriction of $G$ to $\mc C'$, $A$ the identity functor on
  $\mc C'$, and $\theta$ the identity natural transformation. Subsequently
  Proposition \ref{prop.deccorelfunctors} follows from noting that all the
  axioms hold for the corresponding maps in $G\mathrm{Cospan}$.
  
  \paragraph{$\square$ preserves composition.} Suppose we have decorated
  corelations
  \[
    f=(X \stackrel{i_X}{\longrightarrow} N \stackrel{o_Y}{\longleftarrow} Y,
    \enspace 1 \stackrel{s}{\to} FN)
    \qquad
    \mbox{and}
    \qquad 
    g=(Y \stackrel{i_Y}{\longrightarrow} M \stackrel{o_Y}{\longleftarrow} Z,
    \enspace 1 \stackrel{t}{\to} FM)
  \]
  We know the functor $\square$ preserves composition on the cospan part; this
  is precisely the content of Proposition \ref{prop.corelfunctors}. It remains to
  check that $\square( g \circ f)$ and $\square g \circ \square f$ have
  isomorphic decorations. This is expressed by the commutativity of the
  following diagram:
  \[
    \xymatrixrowsep{1.1pc}
    \xymatrixcolsep{.6pc}
    \xymatrix{
      \scriptstyle G\overline{A(\overline{N+_YM})} \ar[rrrrrr]^{Gn} &&&&&&
      \scriptstyle G\overline{(\overline{AN}+_{AY}\overline{AM})}\\
      \\
      \scriptstyle GA(\overline{N+_YM}) \ar[uu]^{Gm^\opp_{A(\overline{N+_YM})}} &&&
      \textsc{\tiny($\ast\ast$)} &&& 
      \scriptstyle G(\overline{AN}+_{AY}\overline{AM})
      \ar[uu]_{Gm^\opp_{\overline{AN}+_{AY}\overline{AM}}} \\
      \\
      \scriptstyle F(\overline{N+_YM}) \ar[uu]^{\theta_{\overline{N+_YM}}} &
      &
      \scriptstyle GA(N+_YM) \ar[uull]_{GAm^\opp_{N+_YM}} && 
      \scriptstyle G(AN+_{AY}AM) \ar[ll]_{G\sim} \ar[uurr]^{G(m_{AN}^\opp
      +_{AY}m_{AM}^\opp)\phantom{spac}}
      & \textsc{\tiny(c)} & 
      \scriptstyle G(\overline{AN}+\overline{AM})
      \ar[uu]_{G[j_{\overline{AN}},j_{\overline{AM}}]} 
      \\
      &\textsc{\tiny(tn)}&& \textsc{\tiny(a)} 
      \\
      \scriptstyle F(N+_YM) \ar[uu]^{Fm_{N+_YM}^\opp}
      &&
      \scriptstyle GA(N+M)\ar[uu]_{GA[j_N,j_M]} && 
      \scriptstyle G(AN+AM) \ar[uu]^{G[j_{AN},j_{AM}]} \ar[ll]^{G\alpha_{N,M}}
      \ar[uurr]^(.6){G(m_{AN}^\opp +m_{AM}^\opp)\phantom{s}} & \textsc{\tiny(gm)} & 
      \scriptstyle G\overline{AN} \times G\overline{AM}
      \ar[uu]_{\gamma_{\overline{AN},\overline{AM}}} \\
      \\
      \scriptstyle F(N+M) \ar[uu]^{F[j_N,j_M]} \ar[uurr]_{\theta_{N+M}} &&&
      \textsc{\tiny(tm)} &&& 
      \scriptstyle GAN \times GAM \ar[uu]_{Gm_{AN}^\opp \times Gm_{AM}^\opp}
      \ar[uull]^{\gamma_{AN,AM}} \\
      \\
      &&&
      \scriptstyle FN \times FM \ar[uulll]^{\varphi_{N,M}}
      \ar[uurrr]_{\theta_N \times \theta_M} \\\\
      &&&
      \scriptstyle 1 \ar[uu]_{\rho_1\circ (s \times t)}
    }
  \]
  This diagram does indeed commute. To check this, first observe that \textsc{(tm)}
  commutes by the monoidality of $\theta$, \textsc{(gm)} commutes by the
  monoidality of $G$, and \textsc{(tn)} commutes by the naturality of $\theta$.
  The remaining three diagrams commute as they are $G$-images of diagrams that
  commute in $\mc C';\mc M'^\opp$. Indeed, \textsc{(a)} commutes since $A$ preserves
  colimits and $G$ is functorial, \textsc{(c)} commutes as it is the $G$-image
  of a pushout square in $\mc C'$, so 
  \[
    \xleftarrow{m_{AN}+m_{AM}}
    \xrightarrow{[j_{\overline{AN}},j_{\overline{AM}}]}
    \quad 
    \textrm{and}
    \quad
    \xrightarrow{[j_{AN},j_{AM}]}
    \xleftarrow{m_{AN}+_{AY}m_{AM}} 
  \]
  are equal as morphisms of $\mc C';\mc M'^\opp$, and \textsc{($\ast\ast$)}
  commutes as it is the $G$-image of the right-hand subdiagram of
  (\ref{diag.eparts}) used to define $n$ in the proof of Lemma
  \ref{lem.corelfuncomposition}.

  \paragraph{Coherence maps are natural.}
  Let $f = (X \longrightarrow N \longleftarrow Y, \enspace 1 \to FN)$, $g= (Z
  \longrightarrow M \longleftarrow W, \enspace 1 \to FM)$ be $F$-decorated
  corelations in $\mc C$. We wish to show that
  \[
    \xymatrixcolsep{4pc}
    \xymatrixrowsep{2pc}
    \xymatrix{
      AX+AY \ar[r]^{\square f+\square g}
      \ar[d]_{\overline{\kappa_{X,Y}}} & 
      AZ+AW \ar[d]^{\overline{\kappa_{Z,W}}} \\
      A(X+Y) \ar[r]^{\square(f+g)} & A(Z+W)
    }
  \]
  commutes in $G\mathrm{Corel}$, where the coherence maps are given by
  \[
    \overline{\kappa_{X,Y}}=          
    \left(
    \begin{aligned}
      \xymatrix{
	& \overline{A(X+Y)} \\  
	AX+AY \ar[ur] && A(X+Y) \ar[ul]
      }
    \end{aligned}
    ,
    \qquad
    \begin{aligned}
      \xymatrixrowsep{1.4ex}
      \xymatrix{
	G(\overline{A(X+Y)}) \\
	GA(X+Y) \ar[u]_{Gm_{AX+AY}^\opp} \\
	G\varnothing \ar[u]_{G!} \\
	1 \ar[u]_{\gamma_1}
      }
    \end{aligned}
    \right).
  \]
  Lemma \ref{lem.corelfunmonoidal} shows that the composites of corelations
  agree. It remains to check that the decorations also agree.

  Here Lemma \ref{lem.emptydecorations} is helpful. Since $\square$ is
  composition preserving, we can replace the $\overline{\kappa}$ with the empty
  decorated coherence maps $\kappa$ of $G\mathrm{Cospan}$, and compute these
  composites in $G\mathrm{Cospan}$, before restricting to the $\mc E'$-parts.
  Lemma \ref{lem.emptydecorations} then implies that the restricted empty
  decorations on the isomorphisms $\overline{\kappa}$ play no role in
  determining the composite decorations. It is thus enough to prove that the
  decorations of $\square f + \square g$ and $\square(f+g)$ are the same up to
  the isomorphism $p\maps G(\overline{AN} +\overline{AM}) \to
  G\overline{A(N+M)}$ between their apices, as defined in the diagram
  (\ref{diag.natural}) in the proof of Lemma \ref{lem.corelfunmonoidal}.

  This comes down to proving the following diagram commutes:
  \[
    \xymatrixrowsep{.8pc}
    \xymatrixcolsep{.8pc}
    \xymatrix{
      &&&& 
      GAN \times GAM \ar[rrdd]^{\gamma} \ar[rr]^{Gm \times Gm} && 
      G\overline{AN} \times G\overline{AM} \ar[rr]^{\gamma} && 
      G(\overline{AN}+\overline{AM}) \ar[dddd]^{Gp}_\sim \\ 
      &&&&&& 
      \textrm{\tiny(G)} \\
      1 \ar[rr]^(.4){\langle s,t \rangle} && 
      FN\times FM \ar[uurr]^{\theta} \ar[ddrr]_{\varphi} && 
      \textrm{\tiny(T)} && 
      G(AN+AM) \ar[uurr]^{G(m+m)} \ar[dd]_{G\kappa} \\
      &&&&&&& 
      \textrm{\tiny(\#\#)}\\ 
      &&&& 
      F(N+M) \ar[rr]_{\theta} && 
      GA(N+M) \ar[rr]_{Gm} && 
      G\overline{A(N+M)}
    }
  \]
  This is straightforward to check: (T) commutes by the monoidality of $\theta$,
  (G) by the monoidality of $G$, and (\#\#) as it is the $G$-image of the
  rightmost square in (\ref{diag.natural}).
\end{proof}

In particular, we get a hypergraph functor from the category of $F$-decorated
cospans to the category of $F$-decorated corelations. In applications, this
is often the key aspect of constructing `black box' or semantic functors.


\begin{corollary}
  Let $\mathcal C$ be a category with finite colimits, and let $(\mathcal E,
  \mathcal M)$ be a factorisation system on $\mathcal C$. Suppose that we also
  have a lax monoidal functor
  \[
    F: (\mathcal C;\mathcal M^\opp,+) \longrightarrow (\Set, \times).
  \]
  Then we may define a category $F\mathrm{Corel}$ with objects the objects of
  $\mathcal C$ and morphisms isomorpism classes of $F$-decorated corelations.

  Write also $F$ for the restriction of $F$ to the wide subcategory $\mathcal
  C$ of $\mathcal C;\mathcal M^\opp$. We can thus also obtain the category
  $F\mathrm{Cospan}$ of
  $F$-decorated cospans. We moreover have a functor 
  \[
    F\mathrm{Cospan} \to F\mathrm{Corel}
  \]
  which takes each object of $F\mathrm{Cospan}$ to itself as an object of
  $F\mathrm{Corel}$, and each decorated cospan
  \[
    \left(
    \begin{aligned}
      \xymatrix{
	& N \\  
	X \ar[ur]^{i} && Y \ar[ul]_{o}
      }
    \end{aligned}
    ,
    \qquad
    \begin{aligned}
      \xymatrix{
	FN \\
	1 \ar[u]_{s}
      }
    \end{aligned}
    \right)
  \]  
  to its jointly $\mathcal E$-part
  \[
    \xymatrix{
      & \overline{N} \\  
      X \ar[ur]^{e\circ \iota_X} && Y \ar[ul]_{e\circ\iota_Y}
    }
  \]
  decorated by the composite
  \[
    \xymatrix{
      1 \ar[r]^s & FN \ar[r]^{Fm_N^\opp} & F\overline{N}.
    }
  \]
\end{corollary}

\section{Examples} \label{sec.excor}
We give two extended examples. Our first example revisits the matrix
example from the introduction, having now developed the material necessary to
formalise it. Our second example is to give two constructions for the category
of linear relations: first as a corelation category, then as a decorated
corelation category.

\subsection{Matrices} \label{ssec.matrices} \hfill

\noindent
Let $R$ be a commutative rig.\footnote{Also known as a semiring, a rig is a ring
  without \emph{n}egatives.} In this subsection we will
construct matrices over $R$ as decorated corelations over $\FinSet^\opp$. 

In $\FinSet^\opp$ the coproduct is the cartesian product $\times$ of sets, the
initial object is the one element set $1$, and cospans are spans in $\FinSet$.
The notation will thus be less confusing if we talk of decorated spans on
$(\FinSet,\times)$ given by the contravariant lax monoidal functor
\begin{align*}
  R^{(-)}: (\mathrm{FinSet},\times) &\longrightarrow (\Set,\times); \\
  N &\longmapsto R^N \\
  \Big(f\maps N \to M\Big) &\longmapsto \Big(R^f\maps R^M \to R^N; v \mapsto v \circ
  f\big).
\end{align*}
The coherence maps $\varphi_{N,M}\maps R^N \times R^M \to R^{N\times M}$ take a
pair $(s,t)$ of maps $s\maps N \to R$, $t\maps M \to R$ to the pointwise product
$s\cdot t \maps N\times M \to R; (n,m) \mapsto s(n) \cdot t(m)$. The unit
coherence map $\varphi_1\maps 1 \to R^1$ sounds almost tautological: it takes
the unique element of the one element set $1$ to the function $1 \to R$ that
maps the unique element of the one element set to the multiplicative identity
$1_R$ of the rig $R$.  As described in the introduction, 
$R^{(-)}\mathrm{Cospan}$ can be considered as the category of `multivalued
matrices' over $R$, and $R^{(-)}\mathrm{Corel}$ the category of matrices over
$R$.

Just as the coherence map $\varphi_1$ gives the unit for the multiplication, it
is the coherence maps $\varphi_{N,M}$ that enact multiplication of scalars:
the composite of decorated spans $(X \xleftarrow{i_X} N \xrightarrow{o_Y} Y,\enspace N
\xrightarrow{s} R)$ and $(Y \xleftarrow{i_Y} M \xrightarrow{o_Z} Z,\enspace M
\xrightarrow{t} R)$ is the span $X \leftarrow N\times_YM \rightarrow Z$ decorated
by the map
\[
  N \times_YM \hooklongrightarrow N \times M \xrightarrow{\varphi_{N,M}(s,t) = s \cdot t}
  R,
\]
where the inclusion from $N \times_YM$ into $N \times M$ is that given by the
categorical product. The intuition for this composition rule, in terms of
channels between elements of $X$ and those of $Z$, was discussed in the
introduction.

As $\varphi_1$ selects the multiplicative unit $1_R$ of $R$, the empty
decoration on any set $N$ is the function that sends every element of $N$ to
$1_R$. This implies the identity decorated span on $X = \{x_1,\dots, x_n\}$ is
that represented by the diagram
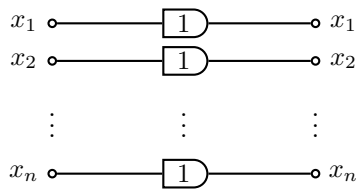
\[
    \tikzset{every path/.style={line width=.8pt}}
\begin{tikzpicture}
	\begin{pgfonlayer}{nodelayer}
		\node [style=sdot] (0) at (1.75, 0.5) {};
		\node [style=sdot] (1) at (-1.75, -1) {};
		\node [style=sdot] (2) at (1.75, -1) {};
		\node [style=none] (3) at (0, -0.25) {$\vdots$};
		\node [style=amp] (4) at (0, -1) {$1$};
		\node [style=amp] (5) at (0, 1) {$1$};
		\node [style=sdot] (6) at (1.75, 1) {};
		\node [style=amp] (7) at (0, 0.5) {$1$};
		\node [style=sdot] (8) at (-1.75, 0.5) {};
		\node [style=sdot] (9) at (-1.75, 1) {};
		\node [style=none] (10) at (1.75, -0.25) {$\vdots$};
		\node [style=none] (11) at (-1.75, -0.25) {$\vdots$};
		\node [style=none] (12) at (-2.125, 1) {$x_1$};
		\node [style=none] (13) at (-2.125, 0.5) {$x_2$};
		\node [style=none] (14) at (-2.125, -1) {$x_n$};
		\node [style=none] (15) at (2.125, 1) {$x_1$};
		\node [style=none] (16) at (2.125, 0.5) {$x_2$};
		\node [style=none] (17) at (2.125, -1) {$x_n$};
	\end{pgfonlayer}
	\begin{pgfonlayer}{edgelayer}
		\draw (6) to (5);
		\draw (4) to (2);
		\draw (1) to (4);
		\draw (5) to (9);
		\draw (8) to (7);
		\draw (7) to (0);
	\end{pgfonlayer}
\end{tikzpicture}
\]
while the Frobenius multiplication and unit are 
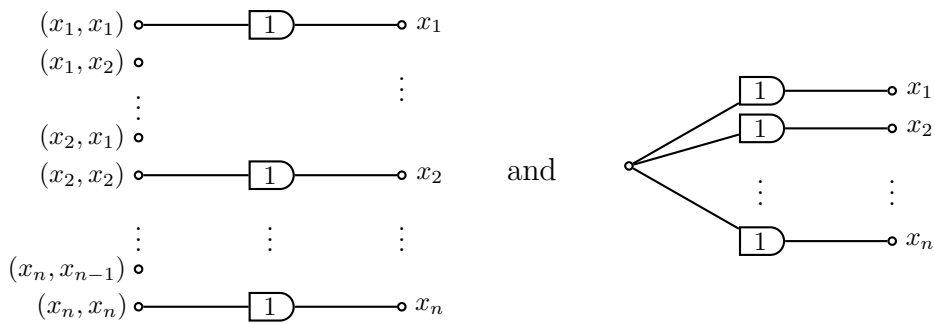
\[
    \tikzset{every path/.style={line width=.8pt}}
  \begin{aligned}
\begin{tikzpicture}
	\begin{pgfonlayer}{nodelayer}
		\node [style=sdot] (0) at (1.75, -0.5) {};
		\node [style=sdot] (1) at (1.75, -2.25) {};
		\node [style=none] (2) at (0, -1.25) {$\vdots$};
		\node [style=amp] (3) at (0, -2.25) {$1$};
		\node [style=amp] (4) at (0, 1.5) {$1$};
		\node [style=sdot] (5) at (1.75, 1.5) {};
		\node [style=amp] (6) at (0, -0.5) {$1$};
		\node [style=sdot] (7) at (-1.75, -0.5) {};
		\node [style=none] (8) at (1.75, -1.25) {$\vdots$};
		\node [style=none] (9) at (-1.75, -1.25) {$\vdots$};
		\node [style=sdot] (10) at (-1.75, -2.25) {};
		\node [style=sdot] (11) at (-1.75, 1.5) {};
		\node [style=none] (12) at (2.125, 1.5) {$x_1$};
		\node [style=none] (13) at (2.125, -0.5) {$x_2$};
		\node [style=none] (14) at (2.125, -2.25) {$x_n$};
		\node [style=none] (15) at (-2.5, 1.5) {$(x_1,x_1)$};
		\node [style=none] (16) at (-2.5, -0.5) {$(x_2,x_2)$};
		\node [style=none] (17) at (-2.5, -2.25) {$(x_n,x_n)$};
		\node [style=none] (18) at (-2.5, 1) {$(x_1,x_2)$};
		\node [style=sdot] (19) at (-1.75, 1) {};
		\node [style=none] (21) at (-1.75, 0.5) {$\vdots$};
		\node [style=none] (22) at (1.75, 0.75) {$\vdots$};
		\node [style=sdot] (23) at (-1.75, -0) {};
		\node [style=none] (24) at (-2.5, -0) {$(x_2,x_1)$};
		\node [style=sdot] (25) at (-1.75, -1.75) {};
		\node [style=none] (26) at (-2.7, -1.75) {$(x_n,x_{n-1})$};
	\end{pgfonlayer}
	\begin{pgfonlayer}{edgelayer}
		\draw (5) to (4);
		\draw (3) to (1);
		\draw (7) to (6);
		\draw (6) to (0);
		\draw (10) to (3);
		\draw (4) to (11);
	\end{pgfonlayer}
\end{tikzpicture}
  \end{aligned}
  \qquad \mbox{and} \qquad
  \begin{aligned}
\begin{tikzpicture}
	\begin{pgfonlayer}{nodelayer}
		\node [style=sdot] (0) at (1.75, 0.5) {};
		\node [style=sdot] (1) at (1.75, -1) {};
		\node [style=none] (2) at (0, -0.25) {$\vdots$};
		\node [style=amp] (3) at (0, -1) {$1$};
		\node [style=amp] (4) at (0, 1) {$1$};
		\node [style=sdot] (5) at (1.75, 1) {};
		\node [style=amp] (6) at (0, 0.5) {$1$};
		\node [style=sdot] (7) at (-1.75, -0) {};
		\node [style=none] (8) at (1.75, -0.25) {$\vdots$};
		\node [style=none] (9) at (2.125, 1) {$x_1$};
		\node [style=none] (10) at (2.125, 0.5) {$x_2$};
		\node [style=none] (11) at (2.125, -1) {$x_n$};
	\end{pgfonlayer}
	\begin{pgfonlayer}{edgelayer}
		\draw (5) to (4);
		\draw (3) to (1);
		\draw (7) to (6);
		\draw (6) to (0);
		\draw (7) to (4);
		\draw (7) to (3);
	\end{pgfonlayer}
\end{tikzpicture}
  \end{aligned}
\]
respectively, with the comultiplication and counit the mirror images.

These morphisms are multivalued matrices in the following sense: the
cardinalities of the domain $X$ and the codomain $Y$ give the dimensions of the
matrix, and the apex $N$ indexes its entries. If $n \in N$ maps to $x \in X$ and
$y \in Y$, we say there is an entry of value $s(n) \in \R$ in the $x$th row and
$y$th column of the matrix. It is multivalued in the sense that there may be
multiple entries in any position $(x,y)$ of the matrix.

To construct matrices proper, and not just multivalued matrices, as decorated
relations, we extend $R^{(-)}$ to the contravariant functor
\[
  R^{(-)}: (\mathrm{Span(FinSet)},\times) \longrightarrow (\Set,\times)
\]
mapping now a span $N \stackrel{f}\leftarrow A \stackrel{g}\to M$ to the
function
\begin{align*}
  R^{f^\opp;g}\maps R^M &\longrightarrow R^N; \\
  v &\longmapsto \Big(n \mapsto \sum_{a \in f^{-1}(n)} v\circ g(a)\Big).
\end{align*}
It is simply a matter of computation to check this is functorial.

Decorated corelations in this category then comprise trivial spans $X
\xleftarrow{\pi_X} X \times Y \xrightarrow{\pi_Y} Y$, where $\pi$ is the
projection given by the categorical product, together with a decoration $X\times
Y \to R$. Such morphisms give a value of $R$ for each pair $(x,y) \in X \times
Y$, and thus are trivially in one-to-one correspondence with $\lvert X \rvert
\times \lvert Y\rvert$-matrices. 

The map $R^{(-)}\mathrm{Cospan} \to R^{(-)}\mathrm{Corel}$ transports the
decoration $N\times_YM \to R$ along the function $N \times_YM \to N \times M$
that identifies elements over the same pair $(x,y)$. In terms of the multivalued
matrices, this sums over (the potentially empty) set of entries over $(x,y)$ to
create a single entry. It is thus easily observed that composition in this
category is matrix multiplication. Moreover, it is not difficult to check that
the monoidal product is the Kroenecker product of matrices, and thus that
$R^{(-)}\mathrm{Corel}$ is monoidally equivalent to the monoidal category of
$(\FinVect, \otimes)$ of finite dimensional vector spaces, linear maps, and the
tensor product.


Note that $R^X$ is always an $R$-module, and $R^f$ a homomorphism of
$R$-modules. Thus we could take decorations here in the category $R\mathrm{Mod}$
of $R$-modules, rather than the category $\Set$. While Proposition
\ref{prop.setdecorations} shows that the resulting decorated cospan and
corelation categories would be isomorphic, this hints at an enriched version of
the theory.

\subsection{Two constructions for linear relations} \hfill

\noindent
We give two constructions for the category of linear relations: first as a
category of epi-mono corelations in the category of linear maps, and second as
isomorphism-morphism corelations in the category sets decorated by linear
subspaces.

Recall that a linear relation $L\maps U \leadsto V$ is a subspace $L \subseteq U
\oplus V$, where $U$, $V$ are vector spaces. We compose linear relations as we
do relations, and vector spaces and linear relations form a category $\LinRel$.
It is straightforward to show that this category can be constructed as the
category of relations in the category $\Vect$ of vector spaces and linear maps
with respect to epi-mono factorisations: monos in $\Vect$ are simply injective
linear maps, and hence subspace inclusions. We show that they may also be
constructed as corelations in $\Vect$ with respect to epi-mono factorisations.

If we restrict to the full subcategory $\FinVect$ of finite dimensional vector
spaces duality makes this easy to see: after picking a basis for each vector
space the transpose yields an equivalence of $\FinVect$ with its opposite
category, so the category of $(\mathcal E,\mathcal M)$-corelations (jointly epic
cospans) is isomorphic to the category of $(\mathcal E,\mathcal M)$-relations
(jointly monic spans) in $\FinVect$. This fact has been fundamental in work on
finite dimensional linear systems and signal flow diagrams \cite{BE,BSZ,BSZ2}.

We prove the general case in detail. To begin, note $\mathrm{Vect}$ has an
epi-mono factorisation system with monos stable under pushouts. This
factorisation system is inherited from $\Set$: the epimorphisms in $\Vect$ are
precisely the surjective linear maps, the monomorphisms are the injective
linear maps, and the image of a linear map is always a subspace of the
codomain, and so itself a vector space. Monos are stable under pushout as the
pushout of a diagram $V \xleftarrow{f} U \xrightarrow{m} W$ is $V \oplus
W/\im[f\; -m]$. The map $m'\maps V \to V \oplus W/\im[f\; -m]$ into the pushout
has kernel $f(\ker m)$. Thus when $m$ is a monomorphism, $m'$ is too.

Thus we have a category of corelations $\corel(\Vect)$. We show that the map
$\corel(\Vect) \to \LinRel$ sending each vector space to itself and each
corelation
\[
  U \stackrel{f}\longrightarrow A \stackrel{g}\longleftarrow V
\]
to the linear subspace $\ker[f\;-g]$ is a full, faithful, and
bijective-on-objects functor.

Indeed, corelations $U \xrightarrow{f} A \xleftarrow{g} V$ are in one-to-one
correspondence with surjective linear maps $U\oplus V \to A$, which are in
turn, by the isomorphism theorem, in one-to-one correspondence with subspaces
of $U\oplus V$. These correspondences are described by the kernel construction
above. Thus our map is evidently full, faithful, and bijective-on-objects. It
also maps identities to identities. It remains to check that it preserves
composition.

Suppose we have corelations $U \xrightarrow{f} A \xleftarrow{g} V$
and $V \xrightarrow{h} B \xleftarrow{\ell} W$. Then their pushout is given by
$P=A \oplus B/\im[g\;-h]$, and we may draw the pushout diagram
\[
  \xymatrix{
    U \ar[dr]_{f} & & V \ar[dl]^{g}  
    \ar[dr]_{h} & & W \ar[dl]^{\ell} {} 
    \\
    & A \ar[dr]_{\iota_A} & & B \ar[dl]^{\iota_B}  \\
    & & P {\save*!<0cm,-.5cm>[dl]@^{|-}\restore}
  }
\]
We wish to show the equality of relations
\[
  \ker[f\;-g];\ker[h\;-\ell] = \ker[\iota_A f\; -\iota_B g].
\]
Now $(\mathbf{u},\mathbf{w}) \in U \oplus W$ lies in the composite relation
$\ker[f\;-g];\ker[h\;-\ell]$ if and only if there exists $\mathbf{v} \in V$ such that
$f\mathbf{u} = g\mathbf{v}$ and $h\mathbf{v} = \ell\mathbf{w}$. But as $P$ is the
pushout, this is true if and only if 
\[
  \iota_A f \mathbf{u} = \iota_A g \mathbf{v} = \iota_B h \mathbf{v} =
  \iota_B \ell \mathbf{w}.
\]
This in turn is true if and only if $(\mathbf{u}, \mathbf{w}) \in \ker[\iota_Af\;
-\iota_B\ell]$, as required. 

This corelational perspective is important as it fits the relational picture
into our philosophy of black boxing. Work by Baez and Erbele, and Bonchi,
Soboci\'nski and Zanasi shows that $\LinRel$ models controllable linear
time-invariant dynamical systems \cite{BE,BSZ,BSZ2}. In \cite{FRS16}, however,
it is shown that it is the construction of $\LinRel$ as corelations, rather than
relations, that correctly generalises to the case of non-controllable systems.
\bigskip

Finally, we give a decorated corelations costruction for $\LinRel$. For
simplicity, we consider just the finite dimensional case. That is, let $\LinRel$
be the category with finite dimensional $k$-vector spaces as objects, and linear
relations as morphisms. We have just seen that this category is a hypergraph
category. Since $\cospan(\FinSet)$ is the theory of special commutative
Frobenius monoids \cite{Lac04}, there exists a hypergraph functor
$\cospan(\FinSet) \to \LinRel$ sending the finite set $1$ to the 1-dimensional
vector space $k$. Also, it is straightforward to check that the covariant hom
functor on the monoidal unit of a symmetric monoidal category is a lax symmetric
monoidal functor; we thus get a functor $\LinRel(0,-)\maps \LinRel \to \Set$. 

Composing these, we have a lax symmetric monoidal functor
\[
  \mathrm{Lin} \maps \big(\mathrm{Cospan}(\mathrm{FinSet}),+\big)
  \longrightarrow \big(\mathrm{Set},\times\big).
\]
This functor takes a finite set $N$ to the set $\mathrm{Lin}(N)$ of linear subspaces of the
vector space $k^N$. Moreover, the image $\mathrm{Lin}(f)$ of a function $f\maps
N \to M$ maps a subspace $L \subseteq k^N$ to $\{v \mid v\circ f \in L\}
\subseteq k^M$, while the image $\mathrm{Lin}(f^\opp)$ of an opposite function
$g^\opp: N \to M$ maps a subspace $L \subseteq k^N$ to $\{v = u \circ g \mid u
\in L\} \subseteq k^M$. 

We thus get a decorated cospan category $\mathrm{LinCospan}$, and a decorated
corelation category $\mathrm{LinCorel}$. The former, $\mathrm{LinCospan}$, has
as morphisms cospans $X \to N \leftarrow Y$ of finite sets decorated by a
subspace of $k^N$. For the latter, note that we take corelations with respect to
the isomorphism-morphism factorisation system $(\mc I_{\FinSet},\FinSet)$. This
means that there is a unique corelation between any two objects; a
representative is simply the cospan $X \to X+Y \leftarrow Y$ given by the
coproduct inclusions. Thus morphisms from $X$ to $Y$ in $\mathrm{LinCorel}$ are
simply subspaces of $k^{X+Y} \cong k^X \oplus k^Y$---that is, linear relations
$k^X \leadsto k^Y$. It is straightforward to check that composition in
$\mathrm{LinCorel}$ is simply relational composition. Thus we have given a
decorated corelation construction for $\LinRel$.

In fact, this method of arriving at a decorated corelation construction applies
to any hypergraph category. The existence of a decorated corelation construction
is useful for the construction of hypergraph functors to and from the category:
it allows such functors to be constructed as decorated corelation functors, and
hence by exhibiting certain natural transformations.

In this particular case, the decorated corelation construction for linear
relations is useful for solving the problem alluded to in the introduction:
constructing semantic functors for electric circuits. Recall that open circuits
themselves have a readily available decorated cospan construction using the
functor $\mathrm{Circ}\maps \FinSet \to \Set$ that maps a finite set $N$ to the
set of circuit diagrams on $N$. Constructing a hypergraph functor from the
resulting decorated cospan category of circuit diagrams to $\LinRel$ is then
simply a matter of finding a monoidal natural transformation from
$\mathrm{Circ}$ to $\mathrm{Lin}\circ \gamma$, where $\gamma\maps \FinSet \to
\cospan(\FinSet)$ is the standard inclusion. This is explored in depth in
\cite{BF,Fon16}.

%
%
%

\appendix
\section{Appendix} \label{sec:proofs}

\paragraph{Decorations in $\Set$ are general.} \label{ssec.setdecorations}

The following observation is due to Sam Staton.

\begin{proposition} \label{prop.setdecorations}
  Let $F\maps (\mathcal C, +) \to (\mathcal D, \otimes)$ be a braided lax monoidal
  functor. Write $\mathcal D(I,-)\maps (\mathcal D, \otimes) \to (\Set,
  \times)$ for the hom functor taking each object $X \in \mathcal
  D$ to the homset $\mathcal D(I,X)$. Then $F\mathrm{Cospan}$ and $\mathcal
  D(I,F-)\mathrm{Cospan}$ are isomorphic as hypergraph categories.
\end{proposition}
\begin{proof}
  Note that the hom functor from the monoidal unit is always lax braided
  monoidal. We have the commutative-by-definition triangle of braided lax
  monoidal functors
  \[
    \xymatrixrowsep{2ex}
    \xymatrix{
      && (\mathcal D,\otimes) \ar[dd]^{\mathcal D(I,-)} \\
      (\mathcal C,+) \ar[urr]^{F} \ar[drr]_{\mathcal D(I,F-)} \\
      && (\Set, \times)
    }
  \]
  By Theorem 4.1 of \cite{Fon15}, this gives rise to a strict hypergraph functor
  $F\mathrm{Cospan} \to \mathcal D(I,F-)\mathrm{Cospan}$. It is easily
  checked that this functor is bijective-on-objects, full, and faithful.
\end{proof}

\end{document}